\documentclass[10pt]{article}

\usepackage{amssymb}
\usepackage{amsmath}
\usepackage{bm}
\usepackage{amsthm}

\def\og{\leavevmode\raise.3ex\hbox{$\scriptscriptstyle\langle\!\langle$~}}
\def\fg{\leavevmode\raise.3ex\hbox{~$\!\scriptscriptstyle\,\rangle\!\rangle$}}

\newcommand{\beqa}{\begin{eqnarray}}
\newcommand{\eeqa}[1]{\label{#1}\end{eqnarray}}
\newcommand{\beq}{\begin{equation}}
\newcommand{\eeq}[1]{\label{#1}\end{equation}}

\newcommand\ep{\varepsilon}

\newcommand\norm[2]{\Vert #1 \Vert_{#2}}

\newcommand\mv[1]{\langle #1 \rangle}

\def\demifleche{\rightharpoonup}

\def\*fleche{\buildrel *\over\demifleche}

\def\tol2{\buildrel\hbox{$L^2$}\over\longrightarrow}

\def\toto{\leaders\hbox to 5mm{\hfil.\hfil}\hfill}

\def\phi{\varphi}

\newtheorem{theorem}{\bf Theorem}[section]

\newtheorem{corollary}{\bf Corollary}[section]
\newtheorem{lemma}{\bf Lemma}[section]

\begin{document}

\title{Resolvent estimates for high-contrast elliptic problems 
with periodic coefficients} 
% The [label2] can be deleted if all authors share the same address

%\selectlanguage{english}
%\author[address1]{A. Burgmann}
\author{
%\large\bf 
K. D. Cherednichenko\footnote{School of Mathematics, Cardiff University, Senghennydd Road, Cardiff, CF24 4AG, United Kingdom. \ \ \ \ \ \ \  \ \ \ \ \ \ \ \   E-mail: CherednichenkoKD@cardiff.ac.uk}\ \ and S. Cooper\footnote{Laboratoire de M\'{e}canique et G\'{e}nie Civil de Montpellier, 860 Rue de Saint-Priest, 34095, Montpellier, France
%Aix-Marseille Universit\'{e}, CNRS, Centrale Marseille, Institut Fresnel, UMR 7249, 13013, Marseille, France
}}

\maketitle

%\centerline{\large\bf K. D. Cherednichenko\footnote{School of Mathematics,Cardiff University, Senghennydd Road, Cardiff, CF24 4AG, United Kingdom} and S. Cooper\footnote{Aix-Marseille Universit\'{e}, CNRS, Centrale Marseille, Institut Fresnel, UMR 7249, 13013, Marseille, France}}

%\

%\

%\author{K. D. Cherednichenko$^{1}$ and S. 
%Cooper$^{2}$}

%\address{$^{1}$School of Mathematics,Cardiff University, Senghennydd Road, Cardiff, CF24 4AG, 
%United Kingdom\\
%$^{2}$Aix-Marseille Universit\'{e}, CNRS, Centrale Marseille, Institut Fresnel, UMR 7249, 13013, Marseille, France}

%\thanks[cor1]{Corresponding author}
%\ead{cherednichenkokd@cardiff.ac.uk}
%\author[address1]{}
%\author[address2]{}
%\address[address1]{}
%\address[address2]{}

% if you know the dates of reception, and acceptation you can put them now;
% idem for the name of the person presenting the Note

%\subject{Analysis of PDE; Spectral Theory; Wave Propagation}

%Keywords: Elliptic differential equations, homogenisation, convergence estimates

%\keywords{Elliptic differential equations, homogenisation, convergence estimates}

%\corres{K, D. Cherednichenko\\ \email{(cherednichenkokd@cardiff.ac.uk)}}

%\medskip

%\begin{center}
%{\small Received *****; accepted after revision +++++\\
%Presented by John Willis}
%\end{center}

\begin{abstract}
We study the asymptotic behaviour of the resolvents $({\mathcal A}^\varepsilon+I)^{-1}$ of elliptic second-order differential 
operators ${\mathcal A}^\varepsilon$ in ${\mathbb R}^d$ with periodic rapidly oscillating 
coefficients, as the period $\varepsilon$ goes to zero. The class of operators covered by our analysis includes both the ``classical'' case
of uniformly elliptic families (where the ellipticity constant does not depend on $\varepsilon$) and the ``double-porosity'' case of coefficients that 
take contrasting values of order one and of order $\varepsilon^2$ in different parts of the period cell.
We provide a construction for the leading order term of the ``operator asymptotics" of $({\mathcal A}^\varepsilon+I)^{-1}$ in the sense of operator-norm convergence and prove order $O(\ep)$ remainder estimates.
\end{abstract}

\section{Introduction}
\label{introduction}

The subject of the present article is the investigation of analytical properties of partial differential equations (PDE) of a special 
kind that emerge in the mathematical theory of homogenisation for periodic composites. The study of composite media has 
been attracting interest since the middle of the last century (see {\it e.g.} $\S 9$ of the monograph \cite{LL},
where some heuristic relationships for the overall properties of mixtures are discussed), although the question of ``averaging'' 
the microstructure in order to get intuitively expected macroscopic quantities goes back a few more decades still. In the early 1970's a number of works have appeared concerning the analysis 
of PDE with periodic rapidly oscillating coefficients, which could be thought of as the simplest, yet already  mathematically challenging,
object representing the idea of a composite structure. For a classical overview of the related developments we refer the reader to the 
books \cite{BLP}, \cite{JKO}.  

In the following years a large amount of literature followed, extending homogenisation theory 
in various directions. One of the central themes of this activity has been in understanding the relative strength of various notions of convergence in terms of characterising the homogenised medium. Unlike in the ``classical" case of uniformly elliptic PDE, whose solutions are compact in the usual Sobolev spaces $W^{l,p}$, non-uniformly elliptic problems offer a variety of descriptions for the homogenised medium that depend on the notion of convergence used. From the computational point of view, one is presented with the question of what approaches yield controlled error estimates for the difference between the original and homogenised solutions.   

A number of results have been obtained recently concerning the difference, in the operator norm, between the resolvent of the 
differential operator representing the original heterogeneous medium 
\beq
-{\rm div}\Bigl(A\Bigl(\frac{x}{\varepsilon}\Bigr)\nabla u\Bigr),\ \ \ \ u\in D^\varepsilon\subset L^2(\Omega),\ \ \ \ \ \varepsilon>0,
\eeq{classop}
and the resolvent of the operator representing the ``homogenisation limit''
\beq
-{\rm div}\bigl(A^{\rm hom}\nabla u\bigr),\ \ \ \ u\in D^{\rm hom}\subset L^2(\Omega).
\eeq{homop}
%representing the ``homogenisation limit''. 
Here $\Omega$ is an open connected subset of ${\mathbb R}^d,$ the matrix function $A$ is $[0,1)^d$-periodic, bounded and 
uniformly positive definite, the constant matrix $A^{\rm hom}$ represents the homogenised medium, and $D^\varepsilon,$ $D^{\rm hom}$ denote the domains of the corresponding operators. 
While a basic order $O(\sqrt{\varepsilon})$ estimate for this setup has been known for a long time, see {\it e.g.} 
\cite{JKO}, one should in principle expect the better rate of convergence of order $O(\varepsilon)$ suggested by the formal asymptotic analysis (assuming that the domain 
$\Omega$ is sufficiently regular). 
The work \cite{BS} contains the related result for problems in the whole space ($\Omega={\mathbb R}^d$), via a combination of  spectral theoretic 
machinery based on the Bloch fibre decomposition of periodic PDE and asymptotic analysis. Earlier works 
\cite{Zhikov1989}, \cite{CV} used similar ideas to prove resolvent convergence, but they did not go as far as getting the order $O(\varepsilon)$
operator norm estimates.  The more recent papers \cite{ZP}, \cite{Kenig} use different techniques to show an improved rate of 
convergence of order $O(\varepsilon\vert\log\varepsilon\vert^\sigma),$ $\sigma>0,$ for problems in bounded domains. 
Finally, the paper \cite{Suslina} combines the earlier results of \cite{BS} with some elements of the approach of \cite{ZP},
for proving the ``expected'' order $O(\varepsilon)$ convergence for such problems.  

The focus of the present paper is on obtaining operator-norm resolvent-type estimates for a class of non-uniformly elliptic problems
%, {\it i.e.} we investigate the behaviour of a class of problems 
of the ``double porosity'' type, where the matrix 
$A=A^\varepsilon$ takes values of order one and of order $\varepsilon^2$ in mutually complementary parts of the ``unit cell'' $[0,1)^d.$  The presence of multiscale effects for such problems was first highlighted in the paper \cite{Allaire}. An analysis of the relation between these effects and the resolvent behaviour of double-porosity problems was carried out in \cite{Zhikov2000}.  

%In relation to the earlier results concerning the resolvents estimates for (\ref{classop})--(\ref{homop}), the main challenge in the case 
%of high-contrast problems 
%%of the kind mentioned above 
%is to capture the limit behaviour of the family of periodic differential operators in $L^2({\mathbb R}^d)$ of the form
%%corresponding to the expressions
%%\beq
%%-\varepsilon^{-2}{\rm div}(A^\varepsilon\nabla u)=
%${\rm div}\bigl((\varepsilon^{-2}\chi+(1-\chi))\nabla u\bigr),$ 
%%\ \ \ \ \ u\in D^{\rm dp}_\varepsilon\subset L^2({\mathbb R}^d),
%%\eeq{scaledBloch}
%at {\it all} frequencies, rather than on a finite interval around zero. Indeed, as $\varepsilon\to0$ the behaviour of 
%the above sequence
%%(\ref{scaledBloch}) 
%on the whole frequency domain $[0,\infty)$ contributes to the leading-order term of the resolvent of (\ref{classop}).
%%behaviour of the original rapidly oscillating family. 

The earlier results (\cite{BS}) concerning resolvent estimates for \eqref{classop}--\eqref{homop} are based on the analysis of spectral projections of the associated operators in a neighbourhood of zero. This approach does not suffice in the double porosity case as all spectral projections provide a 
leading-order contribution to the behaviour of the resolvent as $\varepsilon\to0.$ 
Bearing this in mind, we analyse the asymptotic behaviour of the fibres of the operator provided by the Bloch decomposition. As was observed by \cite{HL}, the pointwise limit of the  fibres is insufficient for norm-resolvent estimates. We show  that in fact the convergence of the individual fibre resolvents is non-uniform with respect to the quasimomentum $\varkappa \in [0,2\pi)^d$. This effect is due to the presence of a ``boundary layer" in the neighbourhood of the origin $\varkappa =0$, where the asymptotics for each fixed $\varkappa$ fails to be valid. To obtain uniform estimates in this neighbourhood we study the asymptotics for the ``rescaled fibres" parametrised by $\theta =\varkappa/\ep.$ The corresponding inner expansion is coupled to the pointwise outer expansion in a matching region where neither expansion is uniform. 

We briefly outline the structure of the paper. In Section 2 we introduce the sequence of problems we analyse. In Section 3 we recall the notions of the direct fibre decomposition and of the  associated Gelfand transform. Section 4  contains the formulation of our main result using these notions. In Section 5 we describe the resolvent asymptotics in the ``inner" region for relevant values of the quasimomentum $\theta\in\varepsilon^{-1}[0,2\pi)^d.$ In Section 6 we introduce spaces $V(\varkappa)\subset H^1_\#(Q),$ $\varkappa\in[0,2\pi)^d,$ which play a key role in our construction. We also prove some lemmas used in the proof of the main result, namely a special Poincar\'{e}-type inequality for the projection on the space orthogonal to $V(\varkappa)$ with respect to the inner product of $H^1_\#(Q),$ as well as several elliptic estimates that are uniform in $\theta.$  Section 7 is devoted to the proof of our main result 
(Theorem \ref{maintheorem}), which consists of two pieces of analysis, in the inner region $\vert\theta\vert\le1$ and in its complement 
$\vert\theta\vert\ge1.$ In Section 8 we discuss the ``outer" region $\vert\theta\vert\ge\varepsilon^{-1/2}$ and show that the inner and outer approximations jointly are only sufficient to obtain a norm-resolvent estimate of order $O(\varepsilon^\alpha),$ $\alpha\in(0,1).$ In Section 9 we calculate the limit of the spectra of the operators $-{\rm div}\bigl(A^\varepsilon(\cdot/\varepsilon)\nabla\bigr)$ and explain its relation to an earlier study of \cite{Zhikov2000}. 
Finally, in Section \ref{examplessection} we show that our main theorem contains as a particular case a result of \cite{BS}, followed by a discussion of  some key points of  the work \cite{Zhikov2000} and the relation of its result to our convergence statement.

% where the notion of the limit itself may somewhat depend on the choice of the convergence notion  

\section{Problem setup}
\label{pset}
In what follows we study the problem
\begin{equation}
\label{maineq}
-{\rm div}\Bigl(A^\varepsilon\Bigl(\frac{x}{\varepsilon}\Bigr)\nabla u\Bigr)+u=f,\ \ \ \ \ f\in L^2({\mathbb R}^d).
\end{equation}
In the above equation 
\[
A^\varepsilon=A_1+\varepsilon^2A_0,
\] 
where  
%$\chi$ is the indicator function of an open set $Q_1\subset Q:=[0,1]^d,$ and
$A_0,$ $A_1$ are $Q$-periodic symmetric 
%non-negative 
$(d\times d)$-matrix functions with entries in $L^\infty(Q).$ 
We assume that $A_0\ge \nu I $, $\nu>0$ and that $A_1\ge\nu I$ on an open set $Q_1\subset Q:=[0,1)^d$ (the ``stiff'' component of the composite) with $A_1=0$  on the interior of $Q\setminus Q_1$ (the ``soft'' component), which we denote by $Q_0.$ 
We also assume that $\overline{Q}_0\subset(0,1)^d,$ which implies, in particular, 
%that $Q_0\neq\emptyset$ and
that the set $\cup_{n\in{\mathbb Z}^d}\bigl(Q_1+n\bigr)$ is connected in ${\mathbb R}^d.$

We next recall the construction of the operator ${\mathcal A}^\varepsilon$ associated with \eqref{maineq}. The closed sesquilinear form 
\[
{\mathfrak a}^\varepsilon(u,v)=\int_{{\mathbb R}^d}A^\varepsilon\Bigl(\frac{x}{\varepsilon}\Bigr)\nabla u(x)\cdot\overline{\nabla v(x)}{\mathrm d}x,\ \ \ \ \ u,v\in H^1({\mathbb R}^d),
\]
is symmetric and non-negative in 
$L^2({\mathbb R}^d),$ hence it generates a self-adjoint operator 
$\mathcal{A}^\ep$ whose domain $D({\mathcal A}^\varepsilon)$ is dense in $L^2({\mathbb R}^d)$ and whose action is described by the identity 
$(\mathcal{A}^\ep u , v)_{L^2({\mathbb R}^d)}= {\mathfrak a}^\varepsilon(u,v)$ for $u \in D({\mathcal A}^\varepsilon)$, $v \in H^1({\mathbb R}^d)$.  The solution $u=u^\varepsilon$ to (\ref{maineq}) is understood as the result of applying the resolvent 
of ${\mathcal A}^\varepsilon$ to $f,$ {\it i.e.} $u_\varepsilon=({\mathcal A}^\varepsilon+I)^{-1}f.$ The last formula is well defined for any $f\in L^2({\mathbb R}^d):$
indeed, the operator ${\mathcal A}^\varepsilon+I$ is clearly bounded below by $I,$ hence it is injective, and the only element $g\in L^2({\mathbb R}^d)$ orthogonal to the image of  ${\mathcal A}^\varepsilon+I$ is $g=0$ by virtue of the fact that the form ${\mathfrak a}^\varepsilon(u,v)+(u,v)_{L^2(\mathbb R^d)},$ $u,v\in H^1({\mathbb R}^d),$ is positive. 
The same fact implies that the resolvent $({\mathcal A}^\varepsilon+I)^{-1}$ is a bounded operator.

Throughout the text we denote by $H^1_\#(Q)$ the space of $Q$-periodic functions that belong to $H^1_{\rm loc}({\mathbb R}^d).$ We use the letter $C$ for any positive constant whose exact value may vary from line to line.
 
%the duality between $H^1_{\#}(Q)$ and its dual space  

%For each $\theta\in Q':=[0,2\pi]^d$ we also 

\section{Bloch formulation and Gelfand transform}
\label{Blochsection}

Using a procedure similar to the above definition of $({\mathcal A}^\varepsilon+I)^{-1},$ for each $\theta \in \ep^{-1} Q'$, where $Q':=[0,2\pi)^d,$ we define 
$u^\ep_{\theta} \in H^1_{\#}(Q)$ as the solution to 
\beq
%\label{bp:eq1}
%\begin{aligned}
%\text{Find $u^\ep_{\theta} \in H^1_{\#}(Q)$ such that} & \\
-\varepsilon^{-2}\left( \nabla+{\rm i}\ep \theta \right) \cdot A^\ep\left( \nabla+{\rm i}\ep \theta \right)  u^\ep_{\theta}+ u^\ep_{\theta}=F,\ \ \ \ F\in L^2(Q),
%\end{aligned} 
\eeq{bp:eq1}
%for a given $F \in L^2$, the space of bounded linear functionals on $H^1_{\#}(Q)$. \
%where $B^\varepsilon:=\varepslion^{-2}A^\varepsilon=\ep^{-2} A_1(y) + A_0(y).$ 
In other words, for all $\theta \in \ep^{-1} Q'$ one has 
$u^\ep_{\theta}=({\mathcal B}_{\varepsilon, \theta}+I)^{-1}F,$ where the operators ${\mathcal B}_{\varepsilon,\theta}$ are generated by the closed 
sesquilinear forms
\[
{\mathfrak b}_{\varepsilon,\theta}(u,v)=\int_{Q}\bigl(\ep^{-2}A_1+A_0\bigr)\bigl(\nabla +{\rm i}\ep\theta\bigr)u\cdot\overline{\bigl(\nabla+{\rm i}\ep \theta\bigr)v}, \ \ \ \ 
u,v\in H^1_{\#}(Q).
\]

\begin{lemma}
\label{Blochdecomplemma}
For each $\varepsilon>0$ there exists a unitary map ${\mathcal U}_\varepsilon: L^2({\mathbb R}^d)\to L^2(\varepsilon^{-1}Q'\times Q)$ such that
\[
{\mathcal U}_\varepsilon({\mathcal A}^\varepsilon+I)^{-1}{\mathcal U}_\varepsilon^{-1}=\int_{\varepsilon^{-1}Q'}^\oplus({\mathcal B}_{\varepsilon,\theta}+I)^{-1}d\theta,
\]
{\it i.e.} for all $f\in L^2({\mathbb R}^d)$ the formula  $({\mathcal A}^\varepsilon+I)^{-1}f={\mathcal U}_\varepsilon^{-1}g$ holds, where for each 
$\theta\in\varepsilon^{-1}Q'$ one has $g(\theta,\cdot)=({\mathcal B}_{\varepsilon,\theta}+I)^{-1}({\mathcal U}_\varepsilon f)(\theta,\cdot).$
\end{lemma}

\begin{proof}
For a given $\varepsilon>0$ set 
\[
({\mathcal U}_\varepsilon f)(\theta,y):=\varepsilon^d\sum_{n\in{\mathbb Z}^d} f\bigl(\varepsilon(y+n)\bigr){\rm e}^{-{\rm i}\varepsilon\theta\cdot(y+n)},
\ \ \ \theta\in\varepsilon^{-1}Q',\  y\in Q.
\]
Note that for each $\varepsilon$ the operator ${\mathcal U}_\varepsilon$ is the composition ${\mathcal T}_\varepsilon{\mathcal G}_\varepsilon$ of a scaled version 
of the usual Gelfand transform ${\mathcal G}_\varepsilon: L^2({\mathbb R}^d)\to L^2(\varepsilon^{-1}Q'\times \varepsilon Q),$ given by
\[
({\mathcal G}_\varepsilon f)(\theta,z):=\varepsilon^{d/2}\sum_{n\in{\mathbb Z}^d} f(z+\varepsilon n){\rm e}^{-{\rm i}\theta\cdot(z+\varepsilon n)}, 
\ \ \ \theta\in\varepsilon^{-1}Q',\ z\in\varepsilon Q,
\]
and the scaling transform ${\mathcal T}_\varepsilon: L^2(\varepsilon^{-1}Q'\times \varepsilon Q)\to 
L^2(\varepsilon^{-1}Q'\times Q)$ given by
\[
({\mathcal T}_\varepsilon h)(\theta,y):=\varepsilon^{d/2}h(\theta, \varepsilon y).
\]
(The inverse ${\mathcal U}_\varepsilon^{-1}$ is the composition ${\mathcal G}_\varepsilon^{-1}{\mathcal T}_\varepsilon^{-1}$ of the inverse of ${\mathcal G}_\varepsilon$ given by 
\[
({\mathcal G}_\varepsilon^{-1}h)(x)=\varepsilon^{d/2}\int_{\varepsilon^{-1}Q'}h(\theta,x){\rm e}^{{\rm i}\theta\cdot x}d\theta, \ \ \ x\in {\mathbb R}^d,
\ \ \ \theta\in\varepsilon^{-1}Q',\ y\in Q.\,
\]
and the inverse of ${\mathcal T}_\varepsilon$ given by
\[
({\mathcal T}_\varepsilon^{-1} h)(\theta,x):=\varepsilon^{-d/2}h(\theta, \tfrac{x}{\ep}). )
\]
The map ${\mathcal U}_\varepsilon$ is unitary since the corresponding property clearly holds for ${\mathcal T}_\varepsilon$ and is well known for ${\mathcal G}_\varepsilon$, 
see {\it e.g.} \cite{BLP}. \end{proof}

\section{Homogenised operator in $\theta$-representation and the main convergence result}
\label{homogsection}

First, we introduce a $\theta$-parametrised operator 
family that plays a  central role in our analysis of the operators ${\mathcal A}^\varepsilon$ as $\varepsilon\to0.$ 

We denote ${\mathcal H}_0:={\mathbb C}\times H^1_0(Q_0),$ and for each $\varepsilon > 0$ and $\theta\in\varepsilon^{-1}Q'$ 
consider the sesquilinear form
\beq
{\mathfrak b}^{\rm hom}_{\varepsilon,\theta}\bigl((c, u),(d,v)\bigr):=A^{\text{hom}}\theta\cdot\theta c\overline{d}+
\int_{Q}A_0 (\nabla+ {\rm i}\varepsilon\theta)u\cdot\overline{(\nabla+ {\rm i}\varepsilon\theta)v},\ \ \ \ \ \ 
%+\int_{Q}(c+v)\overline{(d+\varphi)},
%= \mv{F, d + \varphi}, \\ \quad \forall \varphi \in H^1_0(Q_0), \ \forall d \in \mathbb{C}^d.
(c, u), (d, v)\in{\mathcal H}_0,
%\ \ \ \ \ \ \ \ \ \ \ \ \ \ \ \ \ \ \ \ \ \ \ \ \ \ \ \ \ \ \ \ \ \ \ \ \ \ \ \ \ \ \ \ \ \ \ \ \ \ \ \ \ \ \ \ \ \ \ \ \ \ \ 
\eeq{bhomform}
where $A^{\rm hom}$ is the usual homogenised matrix 
\beq
A^{\rm hom}\xi\cdot\xi=\min_{u\in H^1_{\#}(Q)}\int_{Q}A_1(\xi+\nabla u)\cdot\overline{(\xi+\nabla u)},\ \ \ \xi\in{\mathbb R}^d.
\eeq{matrixAhom}

Note that the matrix $A^{\rm hom}$ is positive definite. Indeed, using the ellipticity assumption on $A_1$ one has, for  $\xi\in{\mathbb R}^d,$ 
\[
A^{\rm hom}\xi\cdot\xi\ge\nu\min_{u\in H^1_{\#}(Q)}\int_{Q_1}\vert\xi+\nabla u\vert^2=\nu\vert\xi\vert^2M(\xi/\vert\xi\vert),
\]
where the function 
\[
M(\eta):=\min_{u\in H^1_{\#}(Q)}\int_{Q_1}\vert\eta+\nabla u\vert^2,\ \ \ \ \ \vert\eta\vert=1,
\]
has a positive minimum $M_{\rm min},$ 
%on the unit sphere $S^{d-1}\subset{\mathbb R}^d,$ 
hence $A^{\rm hom}\ge\nu M_{\rm min}.$

In what follows we also denote 
\[
{\mathcal L}:=\{c+\widetilde{u}: c\in{\mathbb C},\, \widetilde{u}\in L^2(Q),\, \widetilde{u}\vert_{Q_1}=0\}\subset L^2(Q),
\]
and use the invertible ``identification'' map ${\mathcal I}: {\mathbb C}\times L^2(Q_0)\to {\mathcal L}$ that takes each pair $(c, u)$ to the function 
$c+\widetilde{u}\in{\mathcal L}$ with $\widetilde{u}=u$ on $Q_0$ and $\widetilde{u}=0$ on $Q_1.$ 

We next define operators ${\mathcal B}^{\rm hom}_{\varepsilon, \theta}$ in the Hilbert space ${\mathbb C}\times L^2(Q_0)$ 
equipped with the inner product 
$\bigl((c, u), (d, v)\bigr)_0=\bigl({\mathcal I}(c, u),{\mathcal I}(d,v)\bigr)_{L^2(Q)}.$
%(c+\widetilde{u},d+\widetilde{v})_{L^2(Q)},$ 
%$\widetilde{u},$ $\widetilde{v}$ are extensions of $u,$ $v$ by zero outside $Q_0.$ 
These operators are associated, for each value of $\theta\in\varepsilon^{-1}Q',$ with the 
%sesquilinear 
forms ${\mathfrak b}^{\rm hom}_{\varepsilon,\theta}$ by means of the identity
\[
\bigl({\mathcal B}^{\rm hom}_{\varepsilon, \theta}(c,u), (d,v)\bigr)_0={\mathfrak b}^{\rm hom}_{\varepsilon,\theta}\bigl((c,u), (d,v)\bigr),\ \ \ \ (d,v)\in{\mathcal H}_0,
\]
where the pairs $(c,u)$ are taken from the maximal possible domain $D\bigl({\mathcal B}^{\rm hom}_{\varepsilon,\theta}\bigr),$ which can be shown to be dense in ${\mathcal H}_0$ 
and hence in ${\mathbb C}\times L^2(Q_0).$ 

The operators ${\mathcal B}^{\rm hom}_{0,\theta}$ can be viewed, roughly speaking, as the $\theta$-components of the Fourier transform of the two-scale homogenised operator, see Section \ref{examplessection} below, with respect to the ``macroscopic'' variable. However, as we also discuss in the same section, in order to 
obtain operator-norm resolvent estimates it is important to deal with a suitable ``truncation'' of this Fourier transform that restricts the Fourier variable 
$\theta$ to the set $\varepsilon^{-1}Q'.$ From this perspective the analysis below can be viewed as a rigorous procedure for such a truncation. 
Note that in view of the non-uniform behaviour of these truncations as $\varepsilon\to0,$ as we discuss in Section \ref{introduction} and in Section \ref{outersection},  the expression $\varepsilon\theta$ in (\ref{bhomform}) can not be set to zero in the region $\vert\theta\vert\ge 1,$ hence the dependence of the  operators ${\mathcal B}^{\rm hom}_{\varepsilon,\theta}$ on $\varepsilon.$

We also denote by ${\mathcal P}$ 
%the operator product ${\mathcal I}^{-1}{\mathcal P}^\perp,$ where ${\mathcal P}_1$ is 
the orthogonal projection of the Hilbert space $L^2(\varepsilon^{-1}Q'\times Q)$
%=L^2(\varepsilon^{-1}Q')\times L^2(Q)$ 
onto its closed subspace 
\[
\bigl\{c+g: c\in L^2(\varepsilon^{-1}Q'),\, g\in L^2(\varepsilon^{-1}Q'\times Q),\,g(\theta,y)=0\ {\rm a.e.}\ (\theta, y)\in \varepsilon^{-1} Q'\times Q_1\bigr\},
\] 
and by ${\mathcal P}_{\rm f}$ its analogue on each ``fibre'', the orthogonal projection of $L^2(Q)$ onto ${\mathcal L}.$ 
%${\mathcal L}:=\{c+\widetilde{u}: c\in{\mathbb C}, \widetilde{u}=0\ {\rm on}\ Q_1\},$ and ${\mathcal P}_2: {\mathcal L}\to L^2(\varepsilon^{-1}Q')\times L^2(Q_0)$ 
%$takes the function $c+\widetilde{u}$ to the pair $(c,u),$ where $u$ is the restriction of $\widetilde{u}$ to $Q_0.$

%Here we think of  $L^2(Q_0)$ as a subspace of $L^2(Q)$ by the way of the embedding $u\mapsto\widetilde{u},$ where $\widetilde{u}=u$ on $Q_0$ and 
%$\widetilde{u}=0$ on $Q_1.$

The main result of the present paper is as follows.

\begin{theorem}
\label{maintheorem}
%Suppose that the matrix $A^{\rm hom}$ is positive definite. 
The resolvents of the operator family ${\mathcal A}^\varepsilon$ are asymptotically close as $\varepsilon\to 0$ 
%in the norm-resolvent sense, 
to the family
\[
{\mathcal R}^\varepsilon:={\mathcal U}_\varepsilon^{-1}\int_{\varepsilon^{-1}Q'}^\oplus{\mathcal I}\bigl({\mathcal B^{\rm hom}_{\varepsilon,\theta}+I}\bigr)^{-1}{\mathcal I}^{-1}\,{\mathrm d}
\theta\,{\mathcal P}{\mathcal U}_\varepsilon,
\]
where 
the corresponding approximation error is of order $O(\varepsilon).$ More precisely, there exists a constant $C>0,$ independent of $\varepsilon,$ such 
that  
\beq
\bigl\Vert({\mathcal A}^\varepsilon+I)^{-1}-
%-{\mathcal U}_\varepsilon^{-1}\int_{\varepsilon^{-1}Q'}^\oplus\bigl({\mathcal B^{\rm hom}(\theta)+I}\bigr)^{-1}{\mathrm d}\theta\,{\mathcal P}{\mathcal U}_\varepsilon
{\mathcal R}^\varepsilon
\bigr\Vert_{L^2({\mathbb R}^d)\to L^2({\mathbb R}^d)}\le C\varepsilon.
\eeq{mainest}
\end{theorem}

%{\sc Remark.} 
Note that the operator ${\mathcal R}^\varepsilon$ can also be written as
\[
{\mathcal R}^\varepsilon={\mathcal U}_\varepsilon^{-1}\int_{\varepsilon^{-1}Q'}^\oplus{\mathcal I}\bigl({\mathcal B^{\rm hom}_{\varepsilon, \theta}+I}\bigr)^{-1}{\mathcal I}^{-1}
{\mathcal P}_{\rm f}\,{\mathrm d}\theta\,{\mathcal U}_\varepsilon,
\]
which follows from the definitions of the projection operators ${\mathcal P}$ and ${\mathcal P}_{\rm f}.$

\section{The inner expansion and principal term for $\mathcal{B}^{\text{hom}}_{\ep, \theta}$ in the inner region $\vert \theta \vert \le 1$.}
In this section we provide an explicit representation for the behaviour in $\ep$ of the operators $\mathcal{B}^{\rm{hom}}_{\ep, \theta}$ in the region $\vert\theta\vert\le 1$. We refer to this expansion as the inner expansion and to its region of validity as the inner region.

%\subsection{Asymptotics for the inner region in the Bloch space}

Let us consider an asymptotic expansion for solutions to \eqref{bp:eq1} of the form
\beq
u^\ep_\theta =\sum_{n=0}^\infty\varepsilon^n u^{(n)}_\theta, \qquad\qquad u^{(n)}_\theta \in H^1_{\#}(Q),\ \ \ n=0,1,2,...
%u^{(0)}_\theta+\ep u^{(1)}_\theta+\ep^2u^{(2)}_\theta+\ldots, \qquad\qquad u^{(j)}_\theta \in H^1_{\#}(Q).
\eeq{fa:eq1}
Substituting \eqref{fa:eq1} into \eqref{bp:eq1} and comparing the coefficients in front of $\ep^{-2}$ on both sides of the resulting equation 
we find 
\beq
\nabla\cdot A_1\nabla u^{(0)}_\theta = 0,
\eeq{u0eq}
or, equivalently,
\beq
u^{(0)}_\theta \in V : = \left\{u \in H^1_\#(Q)  \ \big\vert \ A_1 \nabla u =0 \right\},
\eeq{fa:eq2}
a space that is naturally isometric to ${\mathcal H}_0$ via the mapping ${\mathcal I}$ defined above:
\[
{\mathcal H_0}\ni(c,v)\mapsto u=c+\widetilde{v}\in V, 
\]
where, as before, $\widetilde{v}=v$ on $Q_0$ and $\widetilde{v}=0$ on $Q_1.$ This implies that $u^{(0)}_\theta =c^{(0)}_\theta+v^{(0)}_\theta$, where the pair $\bigl(c^{(0)}_\theta,v^{(0)}_\theta\bigr)$ belongs to ${\mathcal H}_0$.

Further, comparing the  coefficients in front of $\ep^{-1}$  and using (\ref{fa:eq2}) yields
\beq
-\nabla\cdot A_1\nabla   u^{(1)}_{\theta}={\rm i}\nabla\cdot A_1\theta c^{(0)}_{\theta}.
%\quad \text{in $Q$}, \ \ 
\eeq{fa:eq3}
Introducing ``unit-cell solutions'' $N_k,$ $k=1,...,d,$ that satisfy
\beq
-\sum_{i,j=1}^d\partial_i\bigl((A_1)_{ij}\partial_jN_k\bigr)=\sum_{j=1}^d\partial_j(A_1)_{kj},
%\ \ \ \ \ k=1,...,d,
\eeq{unitcellproblems}
we note that,  up to an arbitrary additive constant, one has
\beq
u^{(1)}_\theta={\rm i}\sum_{j=1}^dN_j\theta_jc^{(0)}_\theta.
\eeq{u1eq}
The concrete choice of the constant added to (\ref{u1eq}) plays an important role in the justification of the asymptotic expansion, which we discuss in Section \ref{s6.2} (see proof of Lemma \ref{auxiliarylemma2}).

Finally, comparing the coefficients in front of $\ep^0$ 
%and using (\ref{fa:eq2}) again 
yields an equation for $u^{(2)}_\theta$ as follows
\beq
-\nabla\cdot A_1\nabla   u^{(2)}_{\theta}= F_{\theta},
\eeq{u2eq}
where
\begin{equation}
\label{F2}
F_{\theta}:=F + {\rm i} \left( \nabla\cdot A_1\theta + \theta \cdot A_1 \nabla \right) u^{(1)}_{\theta}+\nabla\cdot A_0\nabla u^{(0)}_\theta-\theta\cdot A_1\theta c^{(0)}_\theta
-u^{(0)}_\theta.
\end{equation}
Solvability of (\ref{u2eq}) requires that $\bigl\langle F_{\theta}, v\bigr\rangle=0$ for all $v\in V.$ The formula (\ref{u1eq}) and the solvability condition for 
(\ref{u2eq}) imply that 
$u^{(0)}_\theta =c^{(0)}_\theta+v^{(0)}_\theta,$ where the pair $\bigl(c_\theta^{(0)}, v_\theta^{(0)}\bigr)\in{\mathcal H}_0$ satisfies the identity
%\begin{multline}
\begin{equation}
A^{\mathrm{hom}}\theta \cdot \theta c^{(0)}_\theta \overline{d} + \int_{Q} A_0 \nabla v^{(0)}_\theta \cdot \overline{\nabla \varphi}  + \int_Q \bigl( c^{(0)}_\theta + v^{(0)}_\theta\bigr)\overline{(d+\varphi)}=\int_Q {\mathcal P}_{\rm f}F\overline{(d+\varphi)} \\ \quad  \forall (d,\varphi)\in{\mathcal H}_0.
\label{inner.e1}
\end{equation}
%\end{multline}
Following the method outlined in Section \ref{homogsection} for the construction of $\mathcal{B}^{\rm{hom}}_{\ep, \theta}$, we introduce the operator ${\mathcal B}^{\rm hom}_{ 0 ,\theta}$ associated to the problem \eqref{inner.e1} such that $(c^{(0)}_\theta, v^{(0)}_\theta) = \bigl({\mathcal B}^{\rm hom}_{ 0 ,\theta}+I\bigr)^{-1} {\mathcal I}^{-1}{\mathcal P}_{\rm f} F $. The next result shows that ${\mathcal B}^{\rm hom}_{ 0 ,\theta}$ is $\ep$-close in norm to ${\mathcal B}^{\rm hom}_{ \ep ,\theta}$ in the inner region of $\theta.$
% $ \vert \theta \vert \le 1$.
\begin{lemma}
\label{cor1}
There exists $C>0$ such that the estimate
$$
\Bigl\Vert{\mathcal I}\bigl({\mathcal B}^{\rm hom}_{ 0 ,\theta}+I\bigr)^{-1} {\mathcal I}^{-1}{\mathcal P}_{\rm f}  - {\mathcal I}\bigl({\mathcal B}^{\rm hom}_{\varepsilon,\theta}+I\bigr)^{-1}{\mathcal I}^{-1}{\mathcal P}_{\rm f}\Bigr\Vert_{L^2(Q) \rightarrow L^2(Q)} \le C \ep,
$$ 
holds for all $\theta\in\varepsilon^{-1}Q'$ satisfying the inequality $\vert \theta \vert \le 1$.
\end{lemma} 
\begin{proof}
For each value $\theta$ as in the lemma
%such that $\vert \theta\vert<1,$ 
consider the pairs
$(c,v) = (\mathcal{B}^{\rm{hom}}_{0, \theta} + I)^{-1}\mathcal{I}^{-1} P_{\rm{f}}F$ and $(c_\ep,v_\ep) = (\mathcal{B}^{\rm{hom}}_{\ep, \theta} + I)^{-1}\mathcal{I}^{-1} P_{\rm{f}}F,$ that is
\begin{equation}
\label{in.e3}
A^{\mathrm{hom}}\theta \cdot \theta c
%^{(0)}_\theta 
\overline{d} + \int_{Q} A_0 \nabla v
%^{(0)}_\theta 
\cdot \overline{\nabla \varphi}  + \int_Q 
%\bigl
( c
%^{(0)}_\theta 
+ v
%^{(0)}_\theta\bigr
)\overline{(d+\varphi)}=\int_Q {\mathcal P}_{\rm f}F\overline{(d+\varphi)}
%, \\
 \quad  \forall (d,\varphi)\in{\mathcal H}_0,
\end{equation}
and
\begin{equation}
\label{in.e2}
A^{\mathrm{hom}}\theta \cdot \theta c_\ep \overline{d} + \int_{Q} A_0 ( \nabla +{\rm i}\ep \theta)v_\ep \cdot \overline{(\nabla + {\rm i}\ep \theta)\varphi}  + \int_Q \bigl( c_\ep + v_\ep \bigr)\overline{(d+\varphi)}=\int_Q {\mathcal P}_{\rm f}F\overline{(d+\varphi)} \\ \quad  \forall (d,\varphi)\in{\mathcal H}_0.
\end{equation}
By setting $(d, \phi) = (c_\ep,v_\ep)$ in \eqref{in.e2} and noting $v_\ep \in H^1_0(Q_0)$ we arrive at the {\it a priori} bound
\begin{align}
\label{in.e4}
 \norm{\nabla v_\ep}{L^2(Q_0)} & \le C \norm{F}{L^2(Q)}
\end{align}
for some constant $C$.

To prove the result we show that for $u_\ep : = \mathcal{I}(c_\ep,v_\ep)$ and $u: = \mathcal{I}(c,v)$ there exists a constant $C>0$ independent of $\ep, \theta$ such that
\begin{equation}
\label{in.e1}
\norm{u_\ep - u}{H^1(Q)} \le C \ep \vert \theta \vert \norm{F}{L^2(Q)}.
\end{equation}
Subtracting \eqref{in.e3} from \eqref{in.e2} implies 
\begin{multline}
\label{in,ee}
A^{\mathrm{hom}}\theta \cdot \theta ( c_\ep - c) \overline{d} + \int_{Q} A_0  \nabla ( v_\ep - v ) \cdot \overline{\nabla \varphi}  + \int_Q \bigl( c_\ep + v_\ep - c - v \bigr)\overline{(d+\varphi)} \\ =  - \int_Q A_0 \nabla v_\ep \cdot \overline{{\rm i}\ep \theta \varphi} - \int_Q A_0{\rm i }\ep \theta v_\ep \cdot \overline{( \nabla + {\rm i}\ep \theta)\varphi}, \qquad  \forall (d,\varphi)\in{\mathcal H}_0.
\end{multline}
Setting $(d, \phi) =(c_\varepsilon-c, 0)$ in \eqref{in,ee} gives
\begin{flalign*}
\left( A^{\mathrm{hom}}\theta \cdot \theta + 1 \right) ( c_\ep - c)\overline{(c_\ep - c)} \ & =  - (\overline{c_\ep - c}) \int_Q(v_\ep - v) ,
\end{flalign*}
hence
$$
\vert c_\ep - c \vert \le C\bigl\Vert\nabla( v_\ep - v)\bigr\Vert_{L^2(Q_0)},
$$
since $v_\ep, v \in H^1_0(Q_0).$

Setting $(d, \phi) =(c_\varepsilon-c, v_\varepsilon-v)$ in \eqref{in,ee} gives
\begin{flalign*}
\bigl\Vert\nabla( v_\ep - v)\bigr\Vert_{L^2(Q_0)}^2 &\le C   \int_Q A_0 \nabla ( v_\ep - v) \cdot \nabla ( v_\ep - v)  \\ & \le C\left[ - \int_Q A_0 \nabla v_\ep\cdot \overline{{\rm i}\ep \theta (v_\ep - v)}   - \int_Q A_0{\rm i}\ep \theta v_\ep \cdot \overline{( \nabla + {\rm i}\ep \theta)(v_\ep - v)} \right] \\
& \le C \ep \vert \theta \vert \left( \norm{\nabla v_\ep}{L^2(Q_0)}\bigl\Vert\nabla ( v_\ep - v)\bigr\Vert_{L^2(Q)} \right)  .
\end{flalign*}
Taking into account \eqref{in.e4} this implies \eqref{in.e1}, since
$$
\norm{u_\ep - u}{L^2(Q)} \le \vert c_\ep - c\vert + \norm{v_\ep - v}{L^2(Q_0)} \le C \bigl\Vert\nabla( v_\ep - v)\bigr\Vert_{L^2(Q_0)}.
$$
\end{proof}
\section{Auxiliary material}
\subsection{Cell problems}
\label{sec:aux}

One of the key elements in the proof of our main result is the analysis of the properties of the following family of auxiliary ``cell problems'':
\beq
- \nabla  \cdot A_1 \nabla w= G,\ \ \ \  G\in H^{-1}_{\varkappa}(Q):=(H^1_\varkappa(Q))^*. 
\eeq{bp:eq4}
Here $H^1_{\varkappa}(Q),$ $\varkappa\in Q',$ is the space of $\varkappa$ quasi-periodic functions belonging to $H^1(Q)$, {\it i.e.}  $u\in H^1_{\varkappa}(Q)$ if, and only if, $u(y)=\exp({\rm i} \varkappa \cdot y)v(y),$ $y\in Q,$ where\footnote{$H^1_{\varkappa}(Q)$ coincides with $H^1_{\#}(Q)$ when $\varkappa = 0$.} $v\in H^1_{\#}(Q)$.
Note that (\ref{u0eq}), (\ref{fa:eq3}), (\ref{u2eq}) all have the form (\ref{bp:eq4}) for $\varkappa =0$ with $G=0,$ $G={\rm i}\nabla\cdot A_1\theta c^{(0)}_\theta,$ $G=F_{\theta},$ 
respectively.

For a given matrix function $A_1$ we consider the space
%\label{spaceV}
\begin{equation}
\label{vkappa}
V(\varkappa):=\bigl\{ v \in H^1_{\varkappa}(Q) \big\vert A_1\nabla v=0\bigr\}.
\end{equation}
Note that, for $A_1$ satisfying the assumptions prescribed in Section \ref{pset}, we find
$$
V(\varkappa) = \left\{ \begin{matrix} V & \text{ for } \varkappa = 0, \\ H^1_0(Q_0) & \text{ for }\varkappa \neq 0. \end{matrix} \right. 
$$
A criterion for the existence of solutions to \eqref{bp:eq4} 
is given below by a variant of the Lax-Milgram lemma. 
%The space $V (\theta)$ 
%%as defined above 
%and Lemma \ref{bp:eq5} have been considered, for $\theta = 0$, in the works \cite{Cooper}, 
%\cite{SmKam}, where their relevance to two-scale compactness in homogenisation of degenerate PDE is discussed. We also refer the reader to \cite{Cooper}, 
%\cite{SmKam} for a proof the lemma. 

%under the following assumption.

%{\sc Assumption.} Denote by $V^\perp$ the orthogonal complement of $V$ in $H^1_{\#}$. There exists a constant $C >0$ such that
%\beq
%\norm{w}{H^1(Q)}^2 \le C \norm{A_1 \nabla w}{L^2(Q_1)}^2 \quad \quad \forall w \in V^\perp.
%\eeq{bp:eq5}

\begin{lemma}
\label{bp:lem2}
For all $\varkappa\in Q',$ denote by $V^\perp(\varkappa)$ the orthogonal complement of $V(\varkappa)$ in $H^1_{\varkappa}(Q)$. Then, for all values of $\varkappa$:
%\begin{enumerate}
%\item{

(i) There exists a constant $C >0$ independent of $\varkappa$ such that
\beq
\norm{P_{V^\perp(\varkappa)}w}{H^1(Q)}\le C d(\varkappa) \norm{A_1 \nabla w}{L^2(Q)}\quad \quad \forall w \in H^1_{\varkappa}(Q),
\eeq{bp:eq5.1}
where 
$$
d(\varkappa) = \left\{ \begin{matrix} 1 & \text{ for $\varkappa = 0,$} \\ \vert \varkappa \vert^{-1} & \text{ for $\varkappa \neq 0,$} \end{matrix} \right.
$$
and $P_{V^\perp(\varkappa)}$ is the orthogonal projection of $H^1_{\varkappa}(Q)$ onto $V^\perp(\varkappa)$.
%}
%\item{For $\theta \neq 0$. There exists a constant $C >0$ such that
%\beq
%\norm{P_{V^\perp}w}{H^1(Q)}\le \frac{C}{\left\vert \theta \right\vert} \norm{A_1 \nabla w}{L^2(Q)}\quad \quad \forall w \in H^1_{\theta}(Q),
%\eeq{bp:eq5.2}
%where 
%$P_{V^\perp(\theta)}$ is the orthogonal projection of $H^1_{\theta}(Q)$ onto $V^\perp(\theta)$.}
%\item{ 

(ii) There exists a solution $w \in H^1_{\varkappa}(Q)$ to \eqref{bp:eq4} if and only if $\mv{G,\varphi} =0 $ for all $\varphi \in V(\varkappa)$.
%}
%\item{

(iii) Any solution to \eqref{bp:eq4} is unique up to the addition of an element from $V(\varkappa):$ if $w$ satisfies \eqref{bp:eq4} then $w + v$ satisfies \eqref{bp:eq4} for any $v \in V(\varkappa)$, and 
if $w_1$, $w_2$ satisfy \eqref{bp:eq4} then $w_1-w_2 \in V(\varkappa)$. In particular, if $w$ is a solution to \eqref{bp:eq4} then $P_{V^\perp(\varkappa)} w$ is the unique part in $V^\perp(\varkappa)$ of any solution to \eqref{bp:eq4}.
%}
%\end{enumerate}
\end{lemma}
\begin{proof}
(i) The inequality \eqref{bp:eq5.1} holds if there exists a constant $C>0$ such that for all $u \in H^1_{\varkappa}(Q)$ there exists $v \in V(\varkappa)$ such that
\begin{equation*}
\norm{u - v}{H^1(Q)} \le C d(\varkappa) \int_{Q_1} \vert \nabla u \vert^2 \ \mathrm{d}y.
\end{equation*}
We shall now verify this for two distinct cases. Case 1: $\varkappa = 0$. For fixed $u \in H^1_{\#}(Q)$, denote 
$\widetilde{u}\in H^1_{\#}(Q)$ to be an extension of $u$ such that 
\begin{equation*}
\norm{ \nabla \widetilde{u}}{L^2(Q)} \le C\norm{\nabla u}{L^2(Q_1)}.
\end{equation*}
Notice that such an extension exists for connected $Q_1$ ({\it cf.} \cite[Section 3.1]{JKO}). Defining $v : = u - \widetilde{u} + \vert Q_1 \vert^{-1}\int_{Q_1} \widetilde{u},$ we see that $v \in V$ and 
$$
\norm{u - v}{H^1(Q)} = \biggl\Vert\widetilde{u} - \frac{1}{\vert Q_1 \vert}\int_{Q_1} \widetilde{u}\biggr\Vert_{H^1(Q)} \le C \int_Q \vert \nabla\widetilde{u} \vert^2 \le C \int_{Q_1} \vert \nabla u \vert^2,
$$
where the first inequality is a variant of the standard Poincar\'{e} inequality.

Case 2: $\varkappa \neq 0$. For fixed $u \in H^1_\varkappa(Q)$, we show there exists a $v \in V(\varkappa)$ such that
%\begin{equation}
\[
\norm{u - v}{H^1(Q)}^2 \le \frac{C}{\vert \varkappa \vert^2} \norm{\nabla u}{L^2(Q_1)}^2.
\]
%\end{equation}
Denoting the map $\ \widetilde{} \ $ as above, we find $u - \widetilde{u} = : v \in H^1_0(Q_0) \bigl( = V(\varkappa) \bigr)$ and 
\begin{flalign*}
\norm{u - v}{H^1(Q)}^2 = \norm{\widetilde{u}}{H^1(Q)}^2 \le \frac{C}{\vert \varkappa \vert^2} \norm{\nabla \widetilde{u}}{L^2(Q)}^2 \le  \frac{C}{\vert \varkappa \vert^2} \norm{\nabla u}{L^2(Q_1)}^2,
\end{flalign*}
which proves the result. Here we have used the following Poincar\'{e} type inequality
$$
\norm{u}{L^2(Q)}^2 \le \frac{C}{\vert \varkappa \vert^2} \norm{\nabla u}{L^2(Q)}^2 \quad \forall u \in H^1_\varkappa(Q),
$$
which is true since $\vert \varkappa \vert^2$ is the first eigenvalue of the Laplace operator with $\varkappa$-quasiperiodic boundary conditions.

(ii) Let $w$ be a solution of \eqref{bp:eq4} and let $\phi\in V(\varkappa)$. Then,
using the symmetry of $A_1$ and \eqref{vkappa},
\begin{equation}
\left\langle G,\phi\right\rangle=\int_Q A_1\nabla
w\cdot \nabla \phi=\int_Q\nabla w\cdot A_1\nabla \phi=0 \label{fwsymm}
\end{equation}
which yields $\mv{G,\varphi} =0 $ for all $\varphi \in V(\varkappa)$. Conversely, suppose that  $\mv{G,\varphi} =0 $ for all $\varphi \in V(\varkappa),$ and
seek $w\in H_{\varkappa}^1(Q)$ that satisfies
 \eqref{bp:eq4}.
 By \eqref{fwsymm}, the identity 
 \begin{equation}
 \label{eq:pdhom19}
 \int_Q A_1 \nabla w \cdot \nabla \phi = \mv{G,\phi}
 \end{equation}
holds automatically for
all $\phi\in V(\varkappa)$, therefore it is sufficient to verify it for all $\phi\in V(\varkappa)^\bot$.
Seeking $w$ in $V(\varkappa)^\bot$ reduces the problem to showing that, in the Hilbert space $H:=V(\varkappa)^\bot$ with
the norm inherited from $H^1(Q),$ the problem   \eqref{eq:pdhom19} satisfies the conditions of the Lax-Milgram lemma (see {\it e.g.} \cite[Section 1.1]{JKO}). As the bilinear form
\[
B[v,w]:=\int_Q A_1\nabla v\cdot \nabla w
\]
is clearly bounded in $H$, {\it i.e.} for some $C>0$ one has $\bigl\vert B[v,w]\bigr\vert\leq C
\|v\|_{H^1(Q)}\|w\|_{H^1(Q)},$ 
%\ \ \ \forall v,\, w\,\in\,H.
%This follows from \eqref{alinfty}, the fact that $V(\varkappa)$ contains constant
%functions and hence, choosing e.g. as the inner product
%\[
%\left( v, w \right)_{H^1}\,=\,\left(\int_Q v\right)\left(\int_Q w\right)\,+\,
%\int_Q \nabla v \cdot\nabla w,
%\]
%\[
%V(\varkappa)^\bot\,\subset\,\left\{\left. v\in \left(H^1_\#(Q)\right)^n\, \right\vert \,
%\int_Q v\,=\,0\,\right\}.
%\]
%and the Poincar\'e inequality with the mean
in order to satisfy the conditions of the Lax-Milgram lemma it remains to be shown that the form $B$ is coercive, {\it i.e.} for some $\nu>0$ the bound
$B[v,v]\geq\nu\|v\|_{H^1(Q)}^2
%\ \ \ \forall v \in V(\varkappa)^\bot.
$ holds. To this end, note that the boundedness of $A_1$ and \eqref{bp:eq5.1} imply
\[
B[v,v]:=\int_Q A_1\nabla v\cdot \nabla v =
\bigl\|\left(A_1\right)^{1/2}\nabla v\bigr\|_{L^2(Q)}^2\geq
C\left\|A_1\nabla v\right\|_{L^2(Q)}^2\geq
\nu\|v\|_{H^1(Q)}.
\]

Now by the Lax-Milgram lemma, there exists a unique solution $w\in V(\varkappa)^\bot$ to the problem
\[
B[w,\phi]=\langle G,\phi\rangle \ \ \ \forall \phi\,\in\,V(\varkappa)^\bot,
\]
and hence to \eqref{bp:eq4}.

(iii) If $w$ satisfies \eqref{bp:eq4} and $v\in V(\varkappa)$ then
$A_1\nabla v=0$ and hence $w+v$ also satisfies
\eqref{bp:eq4}. Assuming further that $w_1$ and $w_2$ both satisfy \eqref{bp:eq4}, notice that
$v=w_1 - w_2$ is a solution of \eqref{bp:eq4} with $G=0.$ Finally, setting
$\phi = v$ in \eqref{eq:pdhom19} yields
\[
0=\int_Q A_1\nabla v\cdot \nabla v =\bigl\|\left(A_1\right)^{1/2}\nabla v\bigr\|_{L^2(Q)}^2,
\]
 implying that $\left(A_1\right)^{1/2}\nabla v=0$ and hence
$A_1 \nabla v=0$, {\it i.e.} one has $v \in V(\varkappa)$. Assuming now that 
the solutions $w_1, w_2$ are in $V(\varkappa)^\perp,$ the difference $v = w_1 - w_2$ belongs to both $V(\varkappa)$ and $V(\varkappa)^\perp$ and is therefore zero. 
\end{proof}

\begin{corollary}
\label{cor3}
For each $\theta\in\varepsilon^{-1}Q'$ and 
%for each 
$k=1,...,d,$ there exists a unique solution $N_k\in V^\perp$ to the unit-cell problem (\ref{unitcellproblems}). In particular, for any value $c^{(0)} \in \mathbb{C}$, there exists a unique solution 
$u^{(1)}\in V^\perp$ to the problem (\ref{fa:eq3}), for which the estimate 
\beq
\bigl\Vert u^{(1)}\bigr\Vert_{H^1(Q)}\le \Vert N_k\Vert_{H^1(Q)}\vert \theta_k\vert \vert c^{(0)} \vert
\eeq{u1cellest}
holds.
%and the estimates $\Vert N_k\Vert_{H^1(Q)}$

\end{corollary}

\subsection{Elliptic estimates}
\label{s6.2}
In our proof of Theorem \ref{maintheorem} we use the following two statements.

\begin{lemma}
\label{auxiliarylemma}
%Suppose that the matrix $A^{\rm hom}$ is positive definite. 
For each $\theta\in\varepsilon^{-1}Q',$ let $u^{(0)}_\theta=  \mathcal{I}(c^{(0)}_\theta,v^{(0)}_\theta),$ where $\bigl(c^{(0)}_\theta, v^{(0)}_\theta\bigr)$ is the solution to \eqref{inner.e1} with $F\in L^2(Q),$ and let $u^{(1)}_\theta \in H^1_\#(Q)$ be the solution (\ref{u1eq}) to the unit-cell problem 
\eqref{fa:eq3}. Then the following estimates hold with some $C>0:$
\beq
\bigl\vert c^{(0)}_\theta\bigr\vert\le C\bigl(1+\vert\theta\vert^2\bigr)^{-1}\Vert F\Vert_{L^2(Q)},
\eeq{c0es}
\beq
\bigl\Vert u^{(0)}_\theta\bigr\Vert_{H^1(Q)}\le C
%\bigl(1+\vert \theta \vert^2\bigr)^{-1}
\norm{F}{L^2(Q)}, 
\eeq{u0est}
\beq
\bigl\Vert u^{(1)}_\theta\bigr\Vert_{H^1(Q)}\le C\vert\theta\vert\bigl(1+\vert\theta\vert^2\bigr)^{-1}\norm{F}{L^2(Q)}.
\eeq{u1est}
\end{lemma}

\begin{proof}
Setting $(d,\varphi)=(c^{(0)}_\theta, v^{(0)}_\theta)$  in \eqref{inner.e1}, and dropping the scripts ``$(0)$'' and ``$\theta$'' for convenience, yields
%\beq
\[
A^{\text{hom}}\theta\cdot\theta\vert c \vert^2  + \int_Q A_0 \nabla u \cdot \nabla u +  \int_{Q} \vert u \vert^2 = \int_Q {\mathcal P}_{\rm f}F \overline{u},
\]
%\eeq{lem2.e1}
and \eqref{u0est} follows by the Cauchy-Schwarz inequality. 

Setting $(d,\varphi)=(c^{(0)}_\theta, 0)$  in \eqref{inner.e1} yields
%\begin{equation}
%\label{lem2.e02}
\[
A^{\text{hom}}\theta\cdot\theta\vert c\vert^2 + \vert c\vert^2  = \left( \int_Q {\mathcal P}_{\rm f}F \right) \overline{c} - \left( \int_{Q_0} v \right) \overline{c}.
\]
%\end{equation}
Using the estimate
\begin{flalign*}
\left( \int_Q {\mathcal P}_{\rm f}F \right) \overline{c} - \left( \int_{Q_0} v \right) \overline{c} & \le \norm{F}{L^2(Q)} \vert c \vert + \vert Q_0 \vert^{1/2} \norm{v}{L^2(Q)} \vert c \vert \\
& \le\norm{F}{L^2(Q)} \vert c \vert + \vert Q_0\vert^{1/2}\norm{u}{L^2(Q)} \vert c \vert + \vert Q_0\vert^{1/2}\vert c \vert^2,
\end{flalign*}
along with the positivity of $A^{\text{hom}}$ and the bound \eqref{u0est}, 
%and \eqref{lem2.e02}, 
we infer \eqref{c0es}.
The estimate (\ref{u1est}) is now a direct consequence of (\ref{c0es}) and (\ref{u1cellest}).
\end{proof}
\begin{lemma}
\label{auxiliarylemma2}
For each $\theta\in\varepsilon^{-1}Q',$ $ \vert \theta \vert \ge 1 $, let $u^{(0)}_{\ep, \theta} = \mathcal{I}\bigl(c^{(0)}_{\varepsilon,\theta},v^{(0)}_{\varepsilon,\theta}\bigr),$
where 
 the pair  $\bigl(c^{(0)}_{\varepsilon,\theta}, v^{(0)}_{\varepsilon, \theta}\bigr)\in{\mathcal H}_0$ satisfies the identity
\begin{equation}
\label{1inner.e1}
{\mathfrak b}^{\rm hom}_{\varepsilon,\theta}\Bigl(\bigl(c^{(0)}_{\varepsilon,\theta},v^{(0)}_{\varepsilon, \theta}\bigr), (d,\varphi)\Bigr)+\int_Q\bigl(c^{(0)}_{\varepsilon,\theta}+v^{(0)}_{\varepsilon, \theta}\bigr)\overline{(d+\varphi)}=\int_Q {\mathcal P}_{\rm f}F\overline{(d+\varphi)},\ \ \ \ \ (d,\varphi)\in{\mathcal H}_0,
\end{equation}
with $F\in L^2(Q).$
We denote by $u^{(1)}_{\ep, \theta}$ a solution to the unit-cell problem
$$
- \nabla \cdot A_1 \nabla  u^{(1)}_{\ep, \theta} = \mathrm{i} \nabla \cdot A_1 \theta c^{(0)}_{\ep, \theta}
$$
such that
\beq
\int_{Q}A_1\theta\cdot\theta\, u^{(1)}_{\ep, \theta}=0.
\eeq{eq:add1}
Then the following estimates hold with some $C>0:$
\begin{align}
\bigl\vert c^{(0)}_{\ep, \theta} \bigr\vert & \le C \left(1 + \vert \theta \vert \right)^{-2}\Vert F\Vert_{L^2(Q)},
\label{cc0es}\\
\bigl\Vert u^{(0)}_{\ep, \theta}\bigr\Vert_{L^2(Q)} & \le C
%\bigl(1+\vert \theta \vert^2\bigr)^{-1}
\norm{F}{L^2(Q)}, 
\label{uu0est} \\
\bigl\Vert (\nabla + \mathrm{i} \ep \theta )v^{(0)}_{\ep, \theta}\bigr\Vert_{L^2(Q)} & \le C
%\bigl(1+\vert \theta \vert^2\bigr)^{-1}
\norm{F}{L^2(Q)}, 
\label{vv0est} \\
\bigl\Vert u^{(1)}_{\ep, \theta}\bigr\Vert_{H^1(Q)} & \le C\vert\theta\vert\bigl(1+\vert\theta\vert^2\bigr)^{-1}\norm{F}{L^2(Q)}.
\label{uu1est}
\end{align}
\end{lemma}
\begin{proof}
Taking  the unique 
%part of the 
solution $w_{\ep, \theta} \in V^\perp$
%H^1_{\#}(Q)$ 
to the problem
$$
- \nabla \cdot A_1 \nabla  w_{\ep, \theta} = \mathrm{i} \nabla \cdot A_1\theta c^{(0)}_{\ep, \theta},
$$
we find by Corollary \ref{cor3} that
$$
\bigl\Vert w_{\ep, \theta} \bigr\Vert_{H^1(Q)}  \le C \vert \theta \vert \vert c^{(0)}_{\ep, \theta} \vert.
$$
Denoting $u^{(1)}_{\ep, \theta} = w_{\ep, \theta} - \left(  \int_Q A_1 \theta \cdot \theta \right)^{-1}\int_Q A_1 \theta \cdot \theta w_{\ep, \theta}$, it is clear that \eqref{eq:add1} holds. By the properties of  boundedness and ellipticity of $A_1$ we find that 
$$
\left(\int_{Q}A_1\theta\cdot\theta\right)^{-1}\int_{Q}A_1\theta\cdot\theta w_{\ep, \theta} \le C\bigl\Vert w_{\ep, \theta}\bigr\Vert_{L^2(Q)}.
$$ 
In particular, the estimate
\begin{equation*}
\bigl\Vert u^{(1)}_{\ep, \theta} \bigr\Vert_{H^1(Q)}  \le C \vert \theta \vert \vert c^{(0)}_{\ep, \theta} \vert.
\end{equation*}
holds.

Inequalities \eqref{cc0es}--\eqref{uu1est} are now shown by appropriately modifying the proof of Lemma \ref{auxiliarylemma}.
\end{proof}
\begin{lemma}
\label{lem.a1}
Let $\theta\in\varepsilon^{-1}Q',$ and let $F_{\theta}$ be given by \eqref{F2}. There exists a function $R_{\theta} \in H^1_{\#}(Q)$ satisfying
$$
- \nabla \cdot A_1 \nabla R_{\theta}  = F_\theta,
$$
such that
\begin{equation}
\label{R1est}
\bigl\Vert R_{\theta}\bigr\Vert_{H^1(Q)} \le C  \norm{F}{L^2(Q)}
\end{equation}
for some constant $C>0$ independent of $\ep, \theta$.
\end{lemma}
\begin{proof}
The functions $u^{(0)}$ and $u^{(1)}$ are chosen so that $F_{\theta}$ satisfies the solvability condition for the equation \eqref{u2eq}, thus the existence of a solution $u^{(2)}$ is guaranteed by Lemma \ref{bp:lem2}. Denoting by $R_\theta$  to be the unique part in $V^\perp$ of any such solution, 
{\it  i.e.} letting $R_{\theta} \in V^\perp$ be such that
\begin{equation}
\label{le1}
\int_Q A_1 \nabla R_{\theta} \cdot \nabla \phi = \mv{F_\theta,\phi} \qquad \forall \phi \in H^1_{\#}(Q),
\end{equation}
we find, by choosing $\phi = R_{\theta}$ in \eqref{le1} and using the assumptions on $A_1,$ that
$$
\bigl\Vert A_1 \nabla R_{\theta}\bigr\Vert_{L^2(Q)}^2 \le \bigl\Vert A^{1/2}_1\bigr\Vert_{L^\infty(Q)}^2 \bigl\Vert A_1^{1/2} \nabla R_{\theta}\bigr\Vert_{L^2(Q)}^2 \le C\bigl\Vert F_\theta\bigr\Vert_{H^{-1}_{\#}(Q)}\bigl\Vert R_{\theta}\bigr\Vert_{H^1_{\#}(Q)},
$$
where $A^{1/2}_1$ is the square root matrix of $A_1$. Due to Lemma \ref{bp:lem2}(i), it remains to show that
$$
\norm{F_\theta}{H^{-1}_\#(Q)} \le C \norm{F}{L^2(Q)}
$$ 
for some constant $C$. This can be seen by Lemma \ref{auxiliarylemma} and by noting, for $\theta\in\varepsilon^{-1}Q',$ that
\begin{flalign*}
\bigl\vert\mv{F_\theta,\phi}\bigr\vert & =\biggl\vert\int_Q F \overline{\phi} - \mathrm{i} A_1u^{(1)}_\theta  \theta \cdot\overline{\nabla \phi} + \mathrm{i} \theta \cdot A_1 \nabla u^{(1)}_\theta \overline{\phi} - A_0 \nabla u^{(0)}_\theta\cdot\overline{\nabla \phi} - \theta \cdot A_1 \theta c^{(0)}_\theta \overline{\phi} - u^{(0)}_\theta \overline{\phi}\biggr\vert\\
& \le C \left( \norm{F}{L^2(Q)} + \vert \theta \vert\bigl\Vert u^{(1)}_\theta\bigr\Vert_{H^1(Q)} +  \bigl\Vert u^{(0)}_\theta\bigr\Vert_{H^1(Q)} + \vert \theta \vert^2 \bigl\vert c^{(0)}_\theta\bigr\vert \right) \norm{\phi}{H^1(Q)}
\end{flalign*}
for all $\phi\in H^1_\#(Q)$.
\end{proof}
\begin{lemma}
\label{lem.a2}
For each $\varepsilon>0,$ $\theta\neq 0,$ consider $H_{\varepsilon,\theta}\in H^{-1}_{\#}(Q)$ such that 
$\mv{H_{\ep,\theta}, \varphi} = 0$ for all $\varphi \in H^1_0(Q_0).$ Then there exists a solution 
 %a function 
 $R_{\ep,\theta} \in H^1_{\#}(Q)$ to the problem
 $$
- \bigl(\nabla + {\rm i}\ep \theta \bigr) \cdot A_1 \bigl( \nabla +{\rm i}\ep \theta \bigr) R_{\ep,\theta}  = H_{\ep, \theta},
$$
that satisfies the estimate 
%\begin{equation}
%\label{R2est}
\[
\bigl\Vert R_{\ep,\theta}\bigr\Vert_{H^1(Q)} \le C \left[ \frac{1}{\vert \ep \theta \vert}\bigl\Vert H_{\ep,\theta} - \langle H_{\ep, \theta}, 1\rangle\bigr\Vert_{H^{-1}_{\#}(Q)} + \frac{1}{\vert \ep \theta \vert^2} \bigl\vert\langle H_{\ep,\theta},1\rangle\bigr\vert  \right]
\]
%\end{equation}
for some constant $C>0$ independent of $\ep, \theta$.
\end{lemma}
\begin{proof}
For $\varkappa \neq  0$, $V(\varkappa) = H^1_0(Q_0)$, see \eqref{vkappa}.  By Lemma \ref{bp:lem2}, the assumption $\mv{H_{\ep, \theta}, \phi} = 0$ for all $\phi \in H^1_0(Q_0)$ implies there exists a unique weak solution $w_{\ep,\theta}\in V^{\perp}(\ep \theta)$ to the problem 
$$
- \nabla\cdot A_1 \nabla w_{\ep,\theta}(y) = \exp({\rm i}\ep \theta \cdot y)H_{\ep,\theta}(y), \qquad y \in Q,
$$
and
\begin{equation}
\label{le2}
\norm{w_{\ep,\theta}}{H^1(Q)} \le \frac{C}{\vert \ep \theta \vert} \bigl\Vert A_1 \nabla w_{\ep,\theta}\bigr\Vert_{L^2(Q)}.
\end{equation}
As $w_{\ep, \theta}\in H^1_{\ep \theta}(Q)$,  the function $r_{\ep, \theta}(y):=\exp(-{\rm i}\ep\theta\cdot y)w_{\ep, \theta}(y)$ is an element 
 of $H^1_{\#}(Q)$  and satisfies the identity
\begin{equation}
\label{le3}
\int_Q A_1 \bigl( \nabla + {\rm i}\ep \theta \bigr) r_{\ep,\theta}\cdot \overline{ \bigl( \nabla +{\rm i}\ep \theta \bigr) \phi } = \mv{H_{\ep,\theta},\phi} \qquad \forall \phi\in H^1_{\#}(Q),
\end{equation}
and by \eqref{le2} we find that
$$
\norm{r_{\ep, \theta}}{H^1(Q)} \le \frac{C}{\vert \ep \theta \vert^{2}} \norm{H_{\ep, \theta}}{H^{-1}_{\#}(Q)}.
$$
Therefore, by representing $r_{\ep, \theta} = s_{\ep, \theta} + t_{\ep, \theta},$ where $s_{\ep, \theta}, t_{\ep, \theta}$ are the solutions to \eqref{le3} for the right-hand sides $H_{\ep, \theta} - \mv{H_{\ep, \theta},1}$ and $\mv{H_{\ep, \theta},1}$ respectively, we argue that in order to prove the theorem it is sufficient to ensure that there exists a function $R_{\varepsilon,\theta}$ solving \eqref{le3} which satisfies the bound 
%\eqref{R2est} it remains to show that 
\begin{equation}
\label{RR2est}
\norm{R_{\ep, \theta}}{H^1(Q)}\le\frac{C}{\vert \ep \theta \vert} \norm{H_{\ep, \theta}}{H^{-1}_{\#}(Q)},
\end{equation}
for the class of $H_{\ep, \theta} \in H^{-1}_{\#}(Q)$ such that $\mv{H_{\ep, \theta},1}=0$.

Let $w_{\ep, \theta}, r_{\ep, \theta}$ be as above for a given $H_{\ep, \theta} \in H^{-1}_{\#}(Q)$ such that $\mv{H_{\ep, \theta},1}=0$. Denoting by $R_{\ep, \theta}$ an extension of $r_{\ep,\theta}$ such that
\begin{gather}
R_{\ep, \theta} = r_{\ep,\theta} \qquad \text{ in } Q_1, \label{e3} \\
\bigl\Vert \nabla R_{\ep, \theta}\bigr\Vert_{L^2(Q)} \le C \bigl\Vert\nabla r_{\ep,\theta}\bigr\Vert_{L^2(Q_1)}, \label{e4}
\end{gather}
it is clear that $R_{\ep, \theta}$ also satisfies \eqref{le3} with $r_{\varepsilon,\theta}$ replaced by $R_{\varepsilon,\theta}.$ We next show that $R_{\ep, \theta}$ satisfies the inequality  \eqref{RR2est}.
%It is well known that there exists a constant $C>0$ such that
%\begin{equation}
%\label{e5}
%\norm{ u }{H^1(Q)}^2 \le C \left( \left\vert \int_{Q} u \right\vert^2 + \int_{Q} \left\vert \nabla u \right\vert^2 \right), \qquad \forall u \in H^1(Q).
%\end{equation}
Substituting $\phi \equiv 1$ in \eqref{le3}, and recalling that $\mv{H_{\ep, \theta},1} =0$, we infer that
%\begin{flalign*}
\[
\int_Q R_{\ep, \theta}=\biggl(\int_Q \ep^2 A_1 \theta \cdot \theta\biggr)^{-1}\left( - \int_Q \ep^2 A_1 \theta \cdot \theta  \left( R_{\ep, \theta} - \int_Q R_{\ep, \theta} \right)+{\rm i}\ep \int_Q A_1 \nabla R_{\ep,\theta} \cdot \theta \right),
\]
and hence
\[
\biggl\vert\int_Q R_{\ep, \theta}\biggr\vert\le\frac{C}{\left\vert \ep \theta \right\vert^2 } \left( \left\vert \ep \theta \right\vert^2 \biggl\Vert R_{\ep, \theta} - \int_{Q}R_{\ep, \theta}\biggr\Vert_{L^2(Q)}+ \left\vert \ep \theta \right\vert \bigl\Vert\nabla R_{\ep,\theta}\bigr\Vert_{L^2(Q_1)} \right).
\]
%\end{flalign*}
In particular, by \eqref{e4} and the standard Poincar\'{e} inequality, it follows that 
\begin{equation}
\label{e6}
\biggl\vert\int_Q R_{\ep, \theta}\biggr\vert\le \frac{C}{\left\vert \ep \theta \right\vert} \bigl\Vert\nabla r_{\ep,\theta}\bigr\Vert_{L^2(Q_1)}. 
\end{equation}
Therefore, by \eqref{e3}--\eqref{e6} we find that
%\begin{flalign*}
\[
\bigl\Vert R_{\ep, \theta}\bigr\Vert_{H^1(Q)} 
%& 
\le C \left( \left\vert \int_{Q} R_{\ep, \theta} \right\vert^2 + \int_{Q} \bigl\vert \nabla R_{\ep, \theta} \bigr\vert^2 \right)^{1/2}\le \frac{C}{\left\vert \ep\theta \right\vert}\bigl\Vert\nabla r_{\ep, \theta}\bigr\Vert_{L^2(Q_1)}.
\]
%\end{flalign*}
To prove \eqref{RR2est} it now remains to show that
\begin{equation}
\label{e7}
\bigl\Vert\nabla r_{\ep, \theta}\bigr\Vert_{L^2(Q_1)} \le C \bigl\Vert H_{\ep, \theta}\bigr\Vert_{H^{-1}_{\#}(Q)}
\end{equation} 
for some constant $C>0.$
%We shall now show \eqref{e7}. 
%Clearly, by 
By virtue of the 
%obvious 
inequality
$
\Vert r_{\ep, \theta}\Vert_{L^2(Q)} \le\Vert w_{\ep,\theta}\Vert_{H^1(Q)}
$
and \eqref{le2} we find that
%\begin{flalign}
\beq
\bigl\Vert \nabla r_{\ep, \theta}\bigr\Vert_{L^2(Q_1)} 
%&
\le \bigl\Vert\left( \nabla +{\rm i}\ep \theta \right) r_{\ep, \theta}\bigr\Vert_{L^2(Q_1)} + \left\vert \ep\theta \right\vert \norm{ r_{\ep,\theta}}{L^2(Q_1)} 
%\nonumber 
%\\
%& 
\le C  \bigl\Vert \left( \nabla +{\rm i}\ep \theta \right) r_{\ep, \theta}\bigr\Vert_{L^2(Q_1)} 
%\label{ee1}
\eeq{ee1}
%\end{flalign}
Further, substituting $\phi = R_{\ep, \theta}$ in \eqref{le3} and recalling \eqref{e3} yields
%\begin{flalign}
\[
\int_Q{ A_1 \left( \nabla +{\rm i}\ep \theta \right) r_{\ep, \theta} \cdot \overline{\left( \nabla +{\rm  i}\ep \theta \right) r_{\ep, \theta}}}  = \bigl\langle H_{\ep, \theta}, R_{\ep, \theta}\bigr\rangle 
%\nonumber 
=\biggl\langle H_{\ep, \theta}, R_{\ep, \theta}-\int_Q R_{\ep, \theta}\biggr\rangle+ \biggl\langle H_{\varepsilon, \theta}, \int_Q R_{\ep, \theta}\biggr\rangle
%\nonumber 
\]
\beq
= \biggl\langle H_{\ep,\theta}, R_{\ep, \theta} - \int_Q R_{\ep, \theta}\biggr\rangle 
%\nonumber  
\le C \norm{H_{\ep, \theta}}{H^{-1}_{\#}(Q)}\bigl\Vert\nabla R_{\ep, \theta}\bigr\Vert_{L^2(Q)} 
%\nonumber 
\le C \bigl\Vert H_{\ep, \theta}\bigr\Vert_{H^{-1}_{\#}(Q)}\bigl\Vert\nabla r_{\ep, \theta}\bigr\Vert_{L^2(Q_1)}. 
%\label{ee2}
\eeq{ee2}
The last equality above follows from the assumption that $\mv{H_{\ep, \theta},1} = 0$. Finally, inequalities \eqref{ee1} and \eqref{ee2} imply \eqref{e7}. \end{proof}

\section{Proof of the main result}
%Theorem \ref{maintheorem}}
%Justification of the asymptotics (\ref{fa:eq1})}
\label{justification}

In terms of the notation introduced in Sections \ref{Blochsection} and \ref{homogsection},  proving Theorem \ref{maintheorem} is equivalent to showing that there exists a constant $C>0$ independent of $\theta$ and $\varepsilon$ such that
\[
\Bigl\Vert\bigl({\mathcal B}_{\varepsilon, \theta}+1\bigr)^{-1}-{\mathcal I}\bigl({\mathcal B}^{\rm hom}_{\varepsilon, \theta}+1\bigr)^{-1}{\mathcal I}^{-1}{\mathcal P}_{\rm f}\Bigr\Vert_{L^2(Q)\rightarrow L^2(Q)}\le C\varepsilon.
\]
This fact is a consequence of the following theorem.

\begin{theorem}
\label{thetaresest}
For each $\varepsilon>0,$ $\theta\in\varepsilon^{-1}Q',$ let $u^\ep_{\theta}$ be the solution to \eqref{bp:eq1} and let $u_{\varepsilon,\theta}^{(0)}:=c_{\varepsilon,\theta}^{(0)}+v_{\varepsilon,\theta}^{(0)}$ where 
 the pair  $\bigl(c^{(0)}_{\varepsilon,\theta}, v^{(0)}_{\varepsilon, \theta}\bigr)\in{\mathcal H}_0$ satisfies the identity (\ref{1inner.e1}).
%\beq
%{\mathfrak b}^{\rm hom}_{\varepsilon,\theta}\Bigl(\bigl(c^{(0)}_{\varepsilon,\theta},v^{(0)}_{\varepsilon, \theta}\bigr), (d,\varphi)\Bigr)+\int_Q\bigl(c^{(0)}_%{\varepsilon,\theta}+v^{(0)}_{\varepsilon, \theta}\bigr)\overline{(d+\varphi)}=\int_Q {\mathcal P}_{\rm f}F\overline{(d+\varphi)},\ \ \ \ \ (d,\varphi)\in{\mathcal H}_0,
%\eeq{fa:eq100}
%Suppose that the matrix $A^{\rm hom}$ is positive definite. 
%%%%%%%%%%%%%%%%%%%%%%%%%%%%%%%%
%Let, as before, $u^\ep_{\theta}$ and $u^{(0)}_\theta,$ be the solutions to \eqref{bp:eq1} and \eqref{fa:eq4} respectively.  
%%%%%%%%%%%%%%%%%%%%%%%%%%%%%%%%
%we have the following result
%\begin{thm}
%\label{thm:disc}
Then there exists a constant $C>0$ independent of $\theta$ and $\ep$ such that
%\beq
%\[
%\bigl\Vert u^\ep_{\theta}-U^\ep_{\theta}\bigr\Vert_{L^2(Q)} \le C \ep \norm{F}{L^2(Q)}.
%\]
%%\eeq{ja:eq2}
%In particular,
\[
%\beq
\bigl\Vert u^\ep_{\theta}-
%c^{(0)}_{\varepsilon,\theta}-v^{(0)}_{\varepsilon, \theta}
u^{(0)}_{\varepsilon,\theta}
\bigr\Vert_{L^2(Q)} \le C \ep \norm{F}{L^2(Q)}.
\]
%\eeq{ja:eq3}

\end{theorem}
%\end{thm}

\begin{proof} To prove the result we consider $\theta \in \ep^{-1} Q'$ in two regions. 
%\\
%\noindent 
%\textbf{

Case 1:
%} 
$
%\boldsymbol{
\vert \theta \vert \le 1
%}
.$ Let $U^{(1)}_{\ep,\theta} = u^{(0)}_{\theta} + \ep u^{(1)}_{\theta} + \ep^2 R_{\theta},$ where $u^{(0)}_{\theta}$, $u^{(1)}_{\theta}$ and $R_{\theta}$ are given by \eqref{inner.e1}, \eqref{fa:eq3} and Lemma \ref{lem.a1}, respectively.
%Formally calculating $ - \left( \nabla_y + {\rm i}\ep \theta_1 \right) \cdot b_\ep (y)  \left( \nabla_y + i \ep \theta_1 \right) U^\ep_{\theta_1} + U^\ep_{\theta_1} $and equating powers of $\ep$ 
Due to the fact that the functions $u^{(0)}_{\theta}$ and 
$u^{(0)}_{\varepsilon,\theta}$ are $\varepsilon$-close in $L^2(Q)$ uniformly in 
$\theta$ for $\vert \theta \vert \le 1$ ({\it cf.} Lemma \ref{cor1}), it is sufficient to prove that
\[
%\beq
\bigl\Vert u^\ep_{\theta}-
U^{(1)}_{\ep,\theta}
\bigr\Vert_{L^2(Q)} \le C \ep \norm{F}{L^2(Q)}.
\]
By direct calculation we find that 
%the ``error'' 
the difference $z^{(1)}_{\ep, \theta} : = u^\ep_\theta - U^{(1)}_{\ep, \theta}$ is the $H^1_{\#}(Q)$-solution of the equation
%\begin{flalign*}
\begin{equation}
-\varepsilon^{-2}\left( \nabla+ {\rm i}\ep \theta\right) \cdot A^\ep\left(\nabla+ {\rm i}\ep \theta\right) z^{(1)}_{\ep, \theta} + z^{(1)}_{\ep, \theta} = 
F^{(1)}_{\ep,\theta},
\label{z1eq}
\end{equation}
%\end{flalign*}
where the coefficients for the non-positive powers of $\ep$ have cancelled due to the construction of $U^{(1)}_{\ep, \theta}.$ The right-hand side $F^{(1)}_{\ep,\theta}\in H^{-1}_{\#}(Q)$ of (\ref{z1eq}) takes the form
%\beq
\[
F^{(1)}_{\ep,\theta}: = \sum_{n=1}^4 \ep^n T^{(n)}_{\theta},
\]
where
\begin{flalign}
T^{(1)}_{\theta} & : =   {\rm i}\bigl( \nabla \cdot A_1  \theta +  \theta \cdot A_1 \nabla \bigr)R_{\theta} + \bigl(  \nabla \cdot A_0 \nabla - \theta \cdot A_1 \theta - I \bigr) u^{(1)}_\theta  + {\rm i} \bigl( \nabla \cdot A_0  \theta +   \theta \cdot A_0 \nabla \bigr) u^{(0)}_{\theta}, \label{ja:eq4} \\
T^{(2)}_{\theta} & : = \bigl(  \nabla \cdot A_0 \nabla - \theta \cdot A_1 \theta - I \bigr) R_{\theta} + {\rm i} \bigl( \nabla \cdot A_0  \theta +   \theta \cdot A_0 \nabla \bigr) u^{(1)}_\theta - \theta \cdot A_0 \theta u^{(0)}_{\theta}, \label{ja:eq5} \\
T^{(3)}_{\theta} & : =  {\rm i} \bigl( \nabla \cdot A_0  \theta +   \theta \cdot A_0 \nabla \bigr)R_{\theta} - \theta \cdot A_0 \theta u^{(1)}_\theta , \label{ja:eq6} \\
T^{(4)}_{\theta} & : = - \theta \cdot A_0 \theta R_{\theta} \label{ja:eq6.1}
\end{flalign}
are elements of $H^{-1}_{\#} (Q).$
%functionals 
%Since $u^0_\theta, \widetilde{u}^1_\theta \in H^1_\#(Q)$, \eqref{ja:eq3}-\eqref{ja:eq4} show $T^i_\theta \in H^{-1}_{\#}(Q)$ with their action on $\phi \in H^1_\#(Q)$ given by:
%\begin{flalign}
%\mv{T^0_{\theta},\phi} & : = \int_{Q} i \theta \widetilde{u}^1_\theta \cdot \nabla \overline{\phi} - \int_{Q} i \theta \cdot \nabla \widetilde{u}^1_\theta \overline{\phi} + \int_{Q} \theta \cdot a_1 \theta u^0_\theta \overline{\phi}  + \int_{Q}  a_0 \nabla u^0_\theta \cdot \nabla \overline{\phi} +  \int_{Q}  u^0_\theta \overline{\phi}  \label{ja:eq3.1} \\
%\mv{T^1_{\theta},\phi} & : = \int_{Q}  \theta \cdot a_1 \theta \widetilde{u}^1_\theta \overline{\phi}  + \int_{Q}  a_0 \nabla  \widetilde{u}^1_{\theta} \cdot \nabla \overline{\phi}  +\int_{Q} a_0 i \theta u^0_{\theta} \cdot \nabla \overline{\phi}  -  \int_{Q}  i \theta \cdot a_0 \nabla  u^0_{\theta} \overline{\phi}  + \int_{Q}  \widetilde{u}^1{\theta} \overline{\phi} \label{ja:eq4.1}\\
%\mv{T^2_{\theta},\phi} & : = \int_{Q} a_0 i \theta \widetilde{u}^1_{\theta} \cdot \nabla \overline{\phi} - \int_{Q}  i \theta \cdot a_0 \nabla  \widetilde{u}^1_{\theta} \overline{\phi}  + \int_{Q}  \theta \cdot a_0  \theta  u^0_{\theta} \overline{\phi} \label{ja:eq5.1} \\
%\mv{T^3_{\theta},\phi} & : = \int_{Q}  \theta \cdot a_0  \theta  \widetilde{u}^1_{\theta} \overline{\phi} \label{ja:eq6.1}
%\end{flalign}
%\eeq{ja:eq7}
A straightforward calculation shows that equations \eqref{ja:eq4}--\eqref{ja:eq6.1} with the inequalities \eqref{u0est}, \eqref{u1est} and \eqref{R1est} imply the bound
%\beq
\[
\norm{F^{(1)}_{\ep,\theta}}{H^{-1}_{\#}(Q)} \le C \ep \norm{F}{L^2(Q)}.
\]
%\eeq{ja:eq7.1}
Hence, the required inequality
$$
\norm{z^{(1)}_{\ep, \theta}}{L^2(Q)} \le C  \vert \ep\theta \vert \norm{F}{L^2(Q)}.
$$
follows.
%\\

%\noindent 
%\textbf{
Case 2:
%} 
$
%\mathbf{\boldsymbol{
\vert \theta\vert\ge 1 
%\le \ep^{-1}\vert Q'\vert
%}}
.$ 
Let $U^{(2)}_{\ep,\theta} = u^{(0)}_{\varepsilon, \theta} + \ep u^{(1)}_{\ep, \theta} + \ep^2 R_{\ep, \theta},$ where 
%$u^{(0)}_{\varepsilon, \theta}$, 
$u^{(1)}_{\ep, \theta}$ is defined in Lemma \ref{auxiliarylemma2} and $R_{\ep, \theta}$ is given by 
%\eqref{fa:eq100}, \eqref{eq:add1} and 
Lemma \ref{lem.a2} for the right-hand side
$$
H_{\ep, \theta}: = F + {\rm i} \left( \nabla\cdot A_1\theta + \theta \cdot A_1 \nabla \right) u^{(1)}_{\ep, \theta} - \ep \theta \cdot A_1 \theta u^{(1)}_{\ep, \theta} + (\nabla + \mathrm{i} \ep \theta)\cdot A_0(\nabla + \mathrm{i} \ep \theta) v^{(0)}_{\ep, \theta} -\theta\cdot A_1\theta c^{(0)}_{\ep, \theta}-u^{(0)}_{\ep, \theta}.
$$
%respectively
Notice that the following inequalities hold
\begin{align*}
\bigl\Vert H_{\ep,\theta} - \mv{H_{\ep,\theta},1}\bigr\Vert_{H^{-1}_\#(Q)} & \le C \norm{F}{L^2(Q)}, & \bigl\vert \mv{H_{\ep,\theta},1} \bigr\vert & \le C \vert \ep \theta \vert \norm{F}{L^2(Q)},
\end{align*}
for some constant $C>0$ independent of $\ep$ and $\theta$. These follow from Lemma \ref{auxiliarylemma2} and the estimates
\begin{flalign*}
\bigl\vert\mv{H_{\ep, \theta},\phi}\bigr\vert & =\biggl\vert\int_Q \Bigl (F \overline{\phi} - \mathrm{i} A_1 u^{(1)}_{\ep, \theta} \theta\cdot\overline{\nabla \phi} + \mathrm{i} \theta \cdot A_1 \nabla u^{(1)}_{\ep, \theta} \overline{\phi} - \ep \theta \cdot A_1 \theta u^{(1)}_{\ep, \theta} \overline{\phi}\Bigr)\\ & \ \ \ \ \ \ \ \ \ \ \ \ \ \ \ \ \ \ \ \ \ \ \ \ \ \ \ \ \ \ \ \ \ \  - \int_Q\Bigl(A_0 (\nabla + \mathrm{i} \ep \theta) v^{(0)}_{\ep, \theta}\cdot \overline{(\nabla + \mathrm{i} \ep \theta) \phi}  + \theta \cdot A_1 \theta c^{(0)}_{\ep, \theta} \overline{\phi} + u^{(0)}_{\ep, \theta} \overline{\phi}\Bigr)\biggr\vert\\
& \le C \left( \norm{F}{L^2(Q)} + \vert \theta \vert\bigl\Vert u^{(1)}_{\ep, \theta} \bigr\Vert_{H^1(Q)} +  \bigl\Vert (\nabla + \mathrm{i} \ep \theta )v^{(0)}_{\ep, \theta}\bigr\Vert_{L^2(Q_0)} + \vert \theta \vert^2\bigl\vert c^{(0)}_{\ep, \theta}\bigr\vert + \norm{u^{(0)}_{\ep, \theta}}{L^2(Q)}\right) \norm{\phi}{H^1(Q)}
\end{flalign*}
and 
\begin{flalign*}
\bigl\vert\mv{H_{\ep, \theta},1}\bigr\vert & =\biggl\vert \int_Q A_0 (\nabla + \mathrm{i} \ep \theta ) v^{(0)}_{\ep, \theta}\cdot\overline{\mathrm{i} \ep \theta}\biggr\vert \le C \vert \ep \theta \vert \bigl\Vert(\nabla + \mathrm{i}\ep \theta)v^{(0)}_{\ep, \theta}\bigr\Vert_{L^2(Q_0)}
\end{flalign*}
The assumptions of Lemma \ref{lem.a2} hold for $H_{\ep,\theta}$ which implies, along with the above inequalities, that the existence of a function $R_{\ep, \theta} \in H^1_{\#}(Q)$ is guaranteed such that
\begin{align}
\label{R2estt}
\norm{R_{\ep, \theta}}{H^1(Q)} \le C \frac{1}{\vert \ep \theta \vert}\norm{F}{L^2(Q)}.
\end{align}

By direct calculation we find that the ``error'' $z^{(2)}_{\ep, \theta} : = u^\ep_\theta - U^{(2)}_{\ep, \theta}$ is the $H^1_{\#}(Q)$-solution of the equation
\begin{flalign*}
-\varepsilon^{-2}\left( \nabla+ {\rm i}\ep \theta\right) \cdot A^\ep\left(\nabla+ {\rm i}\ep \theta\right) z^{(2)}_{\ep, \theta}  + z^{(2)}_{\ep, \theta}  = 
F^{(2)}_{\ep,\theta},
\end{flalign*}
where the coefficients for the non-positive powers of $\ep$ have cancelled due to the construction of $U^{(2)}_{\ep, \theta}.$ In the above equation the right-hand side $F^{(2)}_{\ep,\theta}\in H^{-1}_{\#}(Q)$ is given by 
%\beq
\[
F^{(2)}_{\ep,\theta}: = \sum_{n=1}^4 \ep^n S^{(n)}_{\theta},
\]
where
\begin{flalign}
S^{(1)}_{\theta} & : =   \bigl(  \nabla \cdot A_0 \nabla  - I \bigr) u^{(1)}_{\ep, \theta} 
+ {\rm i} \bigl( \nabla \cdot A_0  \theta +   \theta \cdot A_0 \nabla \bigr) c^{(0)}_{\ep, \theta}, 
\label{jja:eq4} \\
S^{(2)}_{\theta} & : = \bigl(  \nabla \cdot A_0 \nabla - I \bigr) R_{\ep, \theta} + {\rm i} \bigl( \nabla \cdot A_0  \theta +   \theta \cdot A_0 \nabla \bigr) u^{(1)}_{\ep, \theta}
-\theta \cdot A_0 \theta c^{(0)}_{\ep, \theta}
, \label{jja:eq5} \\
S^{(3)}_{\theta} & : =  {\rm i} \bigl( \nabla \cdot A_0  \theta +   \theta \cdot A_0 \nabla \bigr)R_{\ep, \theta} - \theta \cdot A_0 \theta u^{(1)}_{\ep, \theta}, \label{jja:eq6} \\
S^{(4)}_{\theta} & : = - \theta \cdot A_0 \theta R_{\ep, \theta} \label{jja:eq6.1}
\end{flalign}
are elements of $H^{-1}_{\#} (Q).$ Equations \eqref{jja:eq4}--\eqref{jja:eq6.1} together with inequalities \eqref{cc0es}, \eqref{uu1est} and \eqref{R2estt} imply that
%\beq
\[
\norm{F^{(2)}_{\ep,\theta}}{H^{-1}_{\#}(Q)} \le C \ep \norm{F}{L^2(Q)}.
\]
%\eeq{ja:eq7.1}
Therefore, the bound
$$
\bigl\Vert z^{(2)}_{\ep, \theta}\bigr\Vert_{H^1(Q)} \le C \ep \norm{F}{L^2(Q)}
$$
holds, and the result follows.
\end{proof}

%%%%%%%%%%%%%%%%%%%%%%%%%%%%%%%%%%%%%%%%%%%%%%%%%%%%%%%%%%%%%%%%%%%%%%%%%%%%%%%%%%%%%%%%%%%%%%%%%%%%%%%%%%%%%%%%%%%%%%%%%%%%%%%%%%%%%%%%%%%%%%%%%%%%%%%%%%%%%%%%%%%%%%%%%%%%%%%%%%%%%%%%%%%%%%%%%%%%%%%%%%%%%%%%%%%%%%%%%%%%%%%%%%%%%%%%%%%%%%%%%%%%%%%%%%%%%%%%%%%%%%%%%%%%%%%%%%%%%%%%%%%%%%%%%%%%%%%%%%%%%%%%%%%%%%%%%%%%%%%%%%%%%%%%%%%%%%%%%%%%%%%%%%%%%%%%%%%%%%%%%%%%%%%%%%%%%%%%%%%%%%%%%%%%%%%%%%%%%%%%%%%%%%%%%%%%%%%%%%%%%%%%%%%%%%%%%%%%%%%%%%%%%%%%%%%%%%%%%%%%%%%%%%%%%%%%%%%%%%%%%%%%%%%%%%%%%%%%%%%%%%%%%%%%%%%%%%%%%%%%%%%%%%%%%%%%%%%%%%%%%%%%%%%%%%%%%%%%%%%%%%%%%%%%%%%%%%%%%%%%%%%%%%%%%%%%%%%%%%%%%%%%%%%%%%%%%%%%%%%%%%%%%%%%%%%%%%%%%%%%%%%%%%%%%%%%%%%%%%%%%%%%%%%%%%%%%%%%%%%%%%%%%%%%%%%%%%%%%%%%%%%%%%%%%%%%%%%%%%%%%%%%%%%%%%%%%%%%%%%%%%%%%%%%%%%%%%%%%%%%%%%%%%%%%%%%%%%%%%%%%%%%%%%%%%%%%%%%%%%%%%%%%%%%%%%%%%%%%%%%%%%%%%%%%%%%%%%%%%%%%%%%%%%%%%%%%%%%%%%%%%%%%%%%%%%%%%%%%%%%%%%%%%%%%%%%%%%%%%%%%%%%%%%%%%%%%%%%%

\section{The outer expansion and principal term for $\mathcal{B}^{\text{hom}}_{\ep, \theta}$ in the outer region 
$\vert\theta\vert\ge\ep^{-1/2}
%\le \ep^{-1}\vert Q'\vert
$.}
\label{outersection}

For fixed $\varkappa \neq 0$ we shall study the asymptotics of the following problem: find $w_{\ep,\varkappa}\in H^1_{\varkappa}(Q)$ such that
\begin{equation}
%\begin{aligned}
%\text{} & \\
- \nabla\cdot \left( \ep^{-2} A_1 + A_0 \right) \nabla  w_{\ep,\varkappa}+ w_{\ep,\varkappa} = F,\ \ \ \ \ F \in L^2(Q). 
%\quad \text{in $Q$}, \ \ 
%\end{aligned} 
\label{outer.e1}
\end{equation}
%for given .

Let us consider an asymptotic expansion for the solution to the above problem
%problem \eqref{outer.e1} 
of the form
\beq
w_{\ep,\varkappa}=\sum_{n=0}^\infty\varepsilon^{2n} w^{(n)}_\varkappa, \qquad\qquad w^{(n)}_\varkappa \in H^1_{\varkappa}(Q),\ \ \ n=0,1,2,...
%u^{(0)}_\theta+\ep u^{(1)}_\theta+\ep^2u^{(2)}_\theta+\ldots, \qquad\qquad u^{(j)}_\theta \in H^1_{\#}(Q).
\eeq{outer.e2}
Substituting \eqref{outer.e2} in \eqref{outer.e1} and comparing the coefficients in front of $\ep^{-2}$ on both sides of the resulting equation yields
$$
\nabla\cdot A_1\nabla w^{(0)}_\varkappa = 0,
$$
{\it i.e.} $w^{(0)}_\varkappa \in V(\varkappa)$ or, equivalently, $w^{(0)}_\varkappa \in H^1_0(Q_0)$, see \eqref{vkappa}. Further, comparing the  coefficients in front of $\ep^{0}$ yields
\beq
-\nabla\cdot A_1\nabla   w^{(1)}_{\varkappa}= F + \nabla \cdot A_0 \nabla w^{(0)}_\varkappa - w^{(0)}_\varkappa .
%\quad \text{in $Q$}, \ \ 
\eeq{outer.e3}
The existence of a solution to \eqref{outer.e3} is guaranteed by  Lemma \ref{bp:lem2} if, and only if,  $w^{(0)}_\varkappa$ satisfies the identity
\begin{equation}
\label{outer.e4}
\int_{Q_0}\bigl(A_0 \nabla w^{(0)}_{\varkappa}\cdot\nabla \varphi + w^{(0)}_{\varkappa} \varphi\bigr)= \mv{F, \varphi} \quad \forall \varphi \in H^1_0(Q_0).
\end{equation}
Furthermore, by Lemma \ref{bp:lem2} and \eqref{outer.e3} the unique part of such a solution satisfies the following inequality 
\begin{equation}
\label{outer.ee4}
\bigl\Vert P_{V^{\perp}(\varkappa)}w^{(1)}_{\varkappa}\bigr\Vert_{H^1(Q)} \le \frac{C}{\vert \varkappa \vert^2}   \norm{F}{L^2(Q)},
\end{equation}
for some constant $C$ independent of $\varkappa$.
Existence and uniqueness of $w^{(0)}_\varkappa$ is implied by the ellipticity of $A_0$ in $Q_0$ and standard ellipticity estimates give the following inequality
$$
\norm{w^{(0)}_{\varkappa}}{H^1_0(Q_0)} \le C \norm{F}{L^2(Q)}.
$$
Comparing the powers of $\ep^{2n}$, for $n \ge 1$, yields 
\beq
-\nabla\cdot A_1\nabla   w^{(n+1)}_{\varkappa}= \nabla \cdot A_0 \nabla w^{(n)}_\varkappa - w^{(n)}_\varkappa .
%\quad \text{in $Q$}, \ \ 
\eeq{outer.e5}
The existence of a solution to \eqref{outer.e5} is guaranteed by requiring that $P_{V(\varkappa)} w^{(n)}_\varkappa$ satisfies the identity
\begin{multline}
\label{outer.e6}
\int_{Q_0}\bigl(A_0 \nabla P_{V(\varkappa)} w^{(n)}_\varkappa\cdot \nabla \varphi + P_{V(\varkappa)} w^{(n)}_\varkappa \varphi\bigr)=-\int_{Q_0} \bigl(A_0 \nabla P_{V^\perp(\varkappa)} w^{(n)}_\varkappa\cdot\nabla \varphi - P_{V^\perp(\varkappa)} w^{(n)}_\varkappa \varphi\bigr) \quad \forall \varphi \in H^1_0(Q_0).
\end{multline}
Equation \eqref{outer.e6} implies
$$
\bigl\Vert P_{V(\varkappa)} w^{(n)}_\varkappa\bigr\Vert_{H^1(Q)} \le C \bigl\Vert P_{V^\perp(\varkappa)} w^{(n)}_\varkappa\bigr\Vert_{H^1(Q)}
$$
for some constant $C$. Therefore, by Lemma \ref{bp:lem2} there exists a constant $C>0$ independent of $\varkappa$ such that
\begin{equation*}
\bigl\Vert P_{V^{\perp}(\varkappa)}w^{(n+1)}_{\varkappa}\bigr\Vert_{H^1(Q)} \le \frac{C}{\vert \varkappa \vert^2}\bigl\Vert P_{V^{\perp}(\varkappa)}w^{(n)}_{\varkappa}\bigr\Vert_{H^1(Q)}.
\end{equation*}
In particular, by recalling \eqref{outer.ee4} we find that 
\begin{equation}
\label{outer.ee6}
\norm{w^{(n)}_{\varkappa}}{H^1(Q)} \le \frac{C}{\vert \varkappa \vert^{2n}}   \norm{F}{L^2(Q)}.
\end{equation}
Now constructing the function
$$
U^{(N)}_{\ep,\varkappa} = \sum_{n=0}^{N }\ep^{2n}w^{(n)}_\varkappa\in H^1_{\varkappa}(Q),
$$
we have the following result.

\begin{theorem}
\label{outer.lem1}
Let $w_{\ep,\varkappa}$ be the solution to \eqref{outer.e1}.  Then for any positive integer $N$ there exists a constant $C_N>0$ independent of $\varkappa$ and $\ep$ such that
\begin{equation*}
\bigl\Vert w_{\ep,\varkappa} - U^{(N)}_{\ep,\varkappa}\bigr\Vert_{H^1(Q)} \le C_N\left(\frac{\ep}{\vert \varkappa \vert} \right)^{2N}  \norm{F}{L^2(Q)}.
\end{equation*}
In particular,
%\begin{equation}
\[
\bigl\Vert w_{\ep.\varkappa}- w^{(0)}_{\varkappa}\bigr\Vert_{H^1(Q)} \le  C_N\left(\frac{\ep}{\vert \varkappa \vert} \right)^{2N}  \norm{F}{L^2(Q)}.
\]
%\end{equation}
\end{theorem}
\begin{proof}
Substituting $U^{(N)}_{\ep,\varkappa}$ in to \eqref{outer.e1} and equating powers of $\ep$ yields
$$
- \nabla \cdot  \left( \ep^{-2} A_1 + A_0 \right)  \nabla \bigl( w_{\ep,\varkappa}- U^{(N)}_{\ep,\varkappa} \bigr) + w_{\ep,\varkappa}- U^{(N)}_{\ep,\varkappa}  = \ep^{2N} \bigl( - \nabla \cdot  A_0  \nabla w^{(N)}_{\varkappa}+w^{(N)}_{\varkappa}\bigr).
$$
The results follow by employing \eqref{outer.ee6} and the standard ellipticity estimates.
\end{proof}
Denote by $[g]$ the  multiplication operator for a given function $g$ and denote by $\mathcal{B}_0$ to be the operator associated with the problem \eqref{outer.e4} such that $w^{(0)}_\varkappa = \left( \mathcal{B}_0 + I \right)^{-1}\mathcal{P}_0 F$, where $\mathcal{P}_0$ is the orthogonal projection of $L^2(Q)$ onto $H^1_0(Q_0).$ Theorem \ref{outer.lem1} implies that that $\mathcal{B}_0$ is $\ep$-close to 
${\mathcal B}^{\rm hom}_{\varepsilon,\theta}$ in the region $\vert \theta \vert\ge\ep^{-1/2},$ 
%\le \ep^{-1}\vert Q'\vert,$ 
in the following sense.
\begin{corollary}
\label{cor2}
$$
\bigl\Vert [{\rm e}^{-{\rm i}\ep \theta \cdot}] \left( \mathcal{B}_0 + I \right)^{-1} \mathcal{P}_0\,[{\rm e}^{{\rm i}\ep \theta \cdot}] - \mathcal{I}({\mathcal B}^{\rm hom}_{\varepsilon,\theta}+I)^{-1}\mathcal{I}^{-1}\mathcal{P}_{\rm f} \bigr\Vert_{L^2(Q)} \le C \ep, 
$$ 
for all $\theta\in\varepsilon^{-1}Q'$ such that $\vert \theta\vert\ge\ep^{-1/2}$.
\end{corollary}

Corollary \ref{cor2} and Theorem \ref{thetaresest} imply that in the region 
$\vert \varkappa\vert \ge \ep^{1/2}$ the term $w^{(0)}_\varkappa$ is the principle term in the approximation to 
$w_{\ep,\varkappa}(y) =\exp({\rm i}\varkappa \cdot y)u^\ep_{\ep^{-1} \varkappa}(y),$ $y\in Q,$ in the ``slow" variable $\varkappa$. Further, Lemma \ref{cor1} states that in the region $\vert \theta \vert \le 1$ the function $u^{(0)}_\theta = \mathcal{I}(c^{(0)}_\theta, v^{(0)}_\theta)$ is the principle term in the approximation to $ u^\ep_\theta$ in the ``fast" variable $\theta = \varkappa / \ep$. This leads to the presence of a boundary layer in the Bloch space in the region $1\le \vert \theta\vert \le\varepsilon^{-1/2},$ where neither the ``outer" operator $\mathcal{B}_0$ nor the ``inner" operator $\mathcal{B}^{\text{hom}}_{0, \theta}$ are suitable for order $O(\ep)$ estimates. This leads to an interpretation of $\mathcal{B}^{\text{hom}}_{\ep , \theta}$ as being the non-trivial matching of $\mathcal{B}_{0}$ and $\mathcal{B}^{\text{hom}}_{0, \theta}$ in the boundary layer necessary to achieve order $O(\ep)$ estimates. This interpretation is further supported by the following result, which states that by extending $\mathcal{B}_{0}$ and $\mathcal{B}^{\text{hom}}_{0, \theta}$ in to the boundary layer one can only achieve $O(\ep^\alpha)$ estimates for any $\alpha \in (0, 1)$.
\begin{corollary}
\label{lembl}
For all $\varepsilon>0,$ $\alpha \in (0, 1),$ denote by $B_{\varepsilon^{\alpha-1}}(0)$ the set $\{\theta: \vert\theta\vert<\varepsilon^{\alpha-1}\}.$ The operators
$$
\mathcal{S}^{\ep,\alpha}:=\mathcal{U}^{-1}_\ep \left( \int^\oplus_{\theta \in B_{\ep^{\alpha-1}}(0)} \mathcal{I} \left( \mathcal{B}^{\text{\rm hom}}_{0 , \ep} + I \right)^{-1} \mathcal{I}^{-1} \mathcal{P}_{\rm f} \ \mathrm{d}\theta + \int^\oplus_{\theta \in \ep^{-1}Q' \backslash B_{\ep^{\alpha-1}}(0)}[{\rm e}^{-{\rm i}\ep \theta \cdot}] \left( \mathcal{B}_{0} + I \right)^{-1} \mathcal{P}_0[{\rm e}^{{\rm i}\ep \theta \cdot}]\ \mathrm{d}\theta \right) \mathcal{U}_\ep
$$
are $\ep^\alpha$-close to $\big( \mathcal{A}^\ep + I\big)^{-1}$ in the operator norm, i.e. there exists a constant $C=C(\alpha)$ independent of $\varepsilon$ such that
$$
\bigl\Vert\big( \mathcal{A}^\ep + I\big)^{-1} - \mathcal{S}^{\ep,\alpha}\bigr\Vert_{L^2(\mathbb{R}^d) \rightarrow L^2(\mathbb{R}^d)} \le C \ep^\alpha.
$$

\end{corollary}
%%%%%%%%%%%%%%%%%%%%%%%%%%%%%%%%%%%%%%%%%%%%%%%%%%%%%%%%%%%%%%%%%%%%%%%%%%%%%%%%%%%%%%%%%%%%%%%%%%%%%%%%%%%%%%%%%%%%%%%%%%%%%%%%%%%%%%%%%%%%%%%%%%%%%%%%%%%%%%%%%%%%%%%%%%%%%%%%%%%%%%%%%%%%%%%%%%%%%%%%%%%%%%%%%%%%%%%%%%%%%%%%%%%%%%%%%%%%%%%%%%%%%%%%%%%%%%%%%%%%%%%%%%%%%%%%%%%%%%%%%%%%%%%%%%%%%%%%%%%%%%%%%%%%%%%%%%%%%%%%%%%%%%%%%%%%%%%%%%%%%%%%%%%%%%%%%%%%%%%%%%%%%%%%%%%%%%%%%%%%%%%%%%%%%%%%%%%%%%%%%%%%%%%%%%%%%%%%%%%%%%%%%%%%%%%%%%%%%%%%%%%%%%%%%%%%%%%%%%%%%%%%%%%%%%%%%%%%%%%%%%%%%%%%%%%%%%%%%%%%%%%%%%%%%%%%%%%%%%%%%%%%%%%%%%%%%%%%%%%%%%%%%%%%%%%%%%%%%%%%%%%%%%%%%%%%%%%%%%%%%%%%%%%%%%%%%%%%%%%%%%%%%%%%%%%%%%%%%%%%%%%%%%%%%%%%%%%%%%%%%%%%%%%%%%%%%%%%%%%%%%%%%%%%%%%%%%%%%%%%%%%%%%%%%%%%%%%%%%%%%%%%%%%%%%%%%%%%%%%%%%%%%%%%%%%%%%%%%%%%%%%%%%%%%%%%%%%%%%%%%%%%%%%%%%%%%%%%%%%%%%%%%%%%%%%%%%%%%%%%%%%%%%%%%%%%%%%%%%%%%%%%%%%%%%%%%%%%%%%%%%%%%%%%%%%%%%%%%%%%%%%%%%%%%%%%%%%%%%%%%%%%%%%%%%%%%%%%%%%%%%%%%%%%%%%%%%%%%%%%%%%%%%%%%%%%%%%%%%%%%%%%%%%%%%%%%%%%%%%%%%%%%%%%%%%%
\section{Spectra of the operators ${\mathcal B}^{\rm hom}_{\varepsilon,\theta}$}
Using the definition of the 
%sesquilinear 
form ${\mathfrak b}^{\rm hom}_{\varepsilon,\theta},$ see Section \ref{homogsection}, we infer that a pair $(c,u)\in{\mathcal H}_0$ is an eigenvector of the operator 
${\mathcal B}^{\rm hom}_{\varepsilon,\theta}$ corresponding to an eigenvalue $\lambda$ if and only if
\beq
A^{\text{hom}}\theta\cdot\theta c\overline{d}+\int_{Q}A_0(\nabla+{\rm i}\varepsilon\theta)u\cdot\overline{(\nabla+{\rm i}\varepsilon\theta)v}=\lambda\int_{Q}(c+u)\overline{(d+v)}\ \ \ \ \ \ \ 
%= \mv{F, d + \varphi}, \\ \quad \forall \varphi \in H^1_0(Q_0), \ \forall d \in \mathbb{C}^d.
\forall (d, v)\in{\mathcal H}_0.
%\ \ \ \ \ \ \ \ \ \ \ \ \ \ \ \ \ \ \ \ \ \ \ \ \ \ \ \ \ \ \ \ \ \ \ \ \ \ \ \ \ \ \ \ \ \ \ \ \ \ \ \ \ \ \ \ \ \ \ \ \ \ \ 
\eeq{eigenweak}
%where $A^{\rm hom}$ is the usual homogenised matrix 
%\beq
%A^{\rm hom}\varkappa\cdot\xi=\min_{u\in H^1_{\#}(Q)}\int_{Q}A_1(y)\bigl(\xi+\nabla u(y)\bigr)\cdot\bigl(\xi+\nabla u(y)\bigr)dy,\ \ \ \xi\in{\mathbb R}^d.
%\eeq{matrixAhom}
%We next define operators ${\mathcal B}^{\rm hom}(\theta)$ in the Hilbert space ${\mathbb C}\times L^2(Q_0)$ 
%equipped with the inner product 
%$((u, c), (v,d))_0=(c+u,d+v))_{L^2(Q)}$. These operators are associated, for each value of $\theta\in\varepsilon^{-1}Q',$ with the 
%sesquilinear forms ${\mathfrak b}^{\rm hom}(\theta)$ by means of the identity
%\[
%\bigl({\mathcal B}^{\rm hom}(\theta)(c,u), (d,v)\bigr)_0={\mathfrak b}^{\rm hom}(\theta)\bigl((c,u), (d,v)\bigr),\ \ \ \ (d,v)\in{\mathcal H}_0,
%\] 
Setting $v=0$ in (\ref{eigenweak}) with an arbitrary $d\in{\mathbb C}$ yields
\beq
A^{\text{hom}}\theta\cdot\theta c=\lambda\biggl(c+\int_Qu\biggr).
\eeq{thetaid}
Further, setting $d=0$ in (\ref{eigenweak}) with an arbitrary $v\in H^1_0(Q_0)$ yields
\[
\int_{Q}A_0(\nabla+{\rm i}\varepsilon\theta)u\cdot\overline{(\nabla+{\rm i}\varepsilon\theta)v}=\lambda\int_{Q}(c+u)\overline{v},
\]
from which we deduce that either $\lambda\in S_0:=\{\lambda_j\}_{j=0}^\infty,$ the set of eigenvalues of the operator 
${\mathcal A_0}=-\nabla\cdot A_0\nabla$ in $L^2(Q),$ defined by the sesquilinear form 
\[
{\mathfrak a}_0(u,v):=\int_Q A_0\nabla u\cdot\overline{\nabla v},\ \ \ \ \ \ u,v\in H^1_0(Q_0),
\]
on the maximal possible domain $D({\mathcal A}_0),$ or $\lambda\notin S_0$ and 
\beq
u=\lambda c\sum_{j=0}^\infty(\lambda_j-\lambda)^{-1}\Bigl(\int_{Q_0}\varphi_j^*\Bigr)\overline{\varphi_j^*},
\eeq{uexpr}
where $\varphi^*_j(y):=\varphi_j(y)\exp({\rm i}\varepsilon\theta\cdot y),$ $y\in Q,$ and $\varphi_j$ is the eigenfunction of ${\mathcal A}_0$ corresponding to the eigenvalue $\lambda_j,$ $j=0,1,...$. (We assume that the eigenvalues are ordered 
in the order of magnitude $\lambda_0<\lambda_1\le\lambda_2\le...,$ where multiple  eigenvalues are appear the number of times equal to their multiplicity 
and that $\varphi_j,$ $j=0,1,2,...,$ are real-valued and linearly independent.) In the former case
%,  when $\lambda\in S_0$ and (\ref{deqn}) holds, 
one has $c=0$ and (\ref{thetaid}) implies $\int_{Q}u=0,$ while in the latter case 
$c\in{\mathbb C}$ is arbitrary and by substituting (\ref{uexpr}) into (\ref{thetaid}) one gets
\begin{equation}
A^{\text{hom}}\theta\cdot\theta=\lambda\biggl(1+\lambda\sum_{j=0}^\infty(\lambda_j-\lambda)^{-1}\Bigl\vert\int_{Q_0}\varphi_j^*\Bigr\vert^2\biggr)
%=:\beta(\lambda).
\label{betaeq}
\end{equation}
The expression 
\[
\beta(\lambda) :=\lambda\biggl(1+\lambda\sum_{j=0}^\infty(\lambda_j-\lambda)^{-1}\Bigl(\int_{Q_0}\varphi_j\Bigr)^2\biggr),
\]
obtained by setting $\varepsilon\theta=0$ in the right-hand side of (\ref{betaeq}),
%, considered as a function of $\lambda,$ 
appeared in the work \cite{Zhikov2000}, where 
the behaviour of the spectra of the operators ${\mathcal A}^\varepsilon$ was analysed. In particular, our main theorem above 
(Theorem \ref{maintheorem}) implies the result of \cite{Zhikov} on convergence of the spectra of ${\mathcal A}^\varepsilon,$ as follows.

\begin{theorem}
%Suppose that the matrix $A^{\rm hom}$ is positive definite. 
The spectra of the operators ${\mathcal A}^\varepsilon$ converge in the Hausdorff sense to the union of the
set $S_0$ and the set 
\[
\lim_{\varepsilon\to 0}\bigcup_{\theta\in\varepsilon^{-1}Q'}\bigl\{\lambda: \beta(\lambda)=A^{\text{\rm hom}}\theta\cdot\theta\bigr\}=
\bigl\{\lambda: \beta(\lambda)\ge0\bigr\}.
\]
\end{theorem}
 
%\pagebreak 
%\section{Hausdorff convergence of the spectra of ${\mathcal A}^\varepsilon.$}

\section{Two particular examples of the family $A^\varepsilon$}
\label{examplessection}

Here we discuss two model cases included in our analysis that have emerged in the literature. 
%the classical, ``moderate contrast''  case when the soft component 
%is absent, and the ``double-porosity'' case 
%$A^\varepsilon=\chi\widetilde{A}_0+\varepsilon^2(1-\chi)\widetilde{A}_1,$ where 
%$\widetilde{A}_0,$ $\widetilde{A}_1,$ are constant positive definite mar
%$\chi$ and $1-\chi$ are the indicator functions of the set $Q_1$ and its non-empty complement.

\subsection{Classical homogenisation: $Q_0=\emptyset$}

This is the case when $V$ consists of constant functions on $Q.$ 
The inequality (\ref{bp:eq5.1}) trivially holds for $\kappa \neq 0$ and for $\kappa = 0$ takes the form of the usual Poincar\'{e} inequality for functions with zero mean over $Q.$ Clearly, the space ${\mathcal H}_0$ is isometric to ${\mathbb C}$ and the 
operator family ${\mathcal R}_\varepsilon$ consists of just one element, the resolvent of the usual homogenised operator 
\[
{\mathcal A}^{\rm hom}v:=-\nabla\cdot A^{\rm hom}\nabla,
\]
where the matrix $A^{\rm hom}$ is given by (\ref{matrixAhom}). Indeed, in this example the operator family  
${\mathcal B}^{\rm hom}_{\ep, \theta}$ does not depend on $\varepsilon$ and for each specific value of $\varepsilon$ represents $\theta$-components of the direct fibre decomposition of the operator 
${\mathcal A}^{\rm hom}$ treated as an operator with $\varepsilon$-periodic coefficients, i.e.
$$
\mathcal{U}^{-1}_\ep \mathcal{A}^{\text{hom}} \mathcal{U}_\ep = \int_{\ep^{-1} Q'}^\oplus \theta \cdot A^{\text{hom}} \theta \ \mathrm{d}\theta = \int_{\ep^{-1} Q'}^\oplus \mathcal{I} \mathcal{B}^{\text{hom}}_{\ep, \theta} \mathcal{I}^{-1} \ \mathrm{d}\theta.
$$
Hence in this case Theorem 
\ref{maintheorem} recovers the result of Birman and Suslina \cite{BS} regarding the resolvent convergence estimates for classical 
homogenisation in ${\mathbb R}^d.$

\subsection{The ``double porosity" problem: $Q_0\neq\emptyset,$ $A_0\vert_{Q_1}=0$}

%This is the case when $A_0\chi=0,$ $A_1(1-\chi)=0,$  $A_0+A_1\ge \nu I,$ $\nu>0,$ where $\chi$ is the indicator function of a set 
%$Q_1\subset Q.$ Note that in view of the generall asumptions on $A_0,$ $A_1,$ the set $Q\setminus Q_1$ has non-empty interior, 
%and $\cup_{n\in{\mathbb Z}^d}(Q_1+n)$ is a connected set in ${\mathbb R}^d.$ 
 
%Here the assumption (\ref{bp:eq6}) 
 
This was considered in the work by Zhikov 
\cite{Zhikov}, where the spectrum of double-porosity problems in ${\mathbb R}^d$ was analysed, following an earlier work 
\cite{Zhikov2000} concerning double-porosity models in bounded domains. 

The paper \cite{Zhikov} 
%studies the problem where $A^\varepsilon=\chi I+\varepsilon^2I\  
contains a proof of the strong two-scale convergence of the sequence of solutions $u=u^\varepsilon$ to the problems (\ref{maineq}) to 
the solution $(v_1, v_0)\in {\mathcal H}^{\rm dp}:=H^1({\mathbb R}^d)\times L^2\bigl({\mathbb R}^d,H_0^1(Q_0)\bigr),$ $v_0=v_0(x,y),$ of the problem 
\[
%({\mathcal A}^{\rm db}(v^1,v^0),(\varphi^1,\varphi_0)\bigr)=
{\mathfrak a}^{\rm dp}\bigl((v_1,v_0),(\varphi_1,\varphi_0)\bigr)+\int_{{\mathbb R}^d}\int_Q(v_1+v_0)\overline{(\varphi_1+\varphi_0)}=\int_{{\mathbb R}^d}\int_Q 
f\overline{(\varphi_1+\varphi_0)}, 
\]
where the form ${\mathfrak a}^{\rm dp},$ with $D({\mathfrak a}^{\rm dp})={\mathcal H}^{\rm dp},$ is given by
\[
{\mathfrak a}^{\rm dp}\bigl((v_1,v_0),(\varphi_1,\varphi_0)\bigr):=\int_{{\mathbb R}^d}A^{\rm hom}\nabla v_1\cdot\overline{\nabla\varphi_1}+\int_{{\mathbb R}^d}\int_Q A_0\nabla_yv_0\cdot\overline{\nabla_y\varphi_0}.
\]
The author of \cite{Zhikov} refers to the operator ${\mathcal A}^{\rm dp}$ generated by ${\mathfrak a}^{\rm db}$ as the 
homogenised operator for the family ${\mathcal A}^\varepsilon$ and proves that the spectra of ${\mathcal A}^\varepsilon$ converge to the spectrum of 
${\mathcal A}^{\rm dp}$ as $\varepsilon\to0.$ For continuous right-hand sides $f$ the strong two-scale convergence result of 
\cite{Zhikov2000} implies that 
\beq
\Bigl\Vert u^\varepsilon-v_1(x)-\widetilde{v}_0\Bigl(x,\frac{x}{\varepsilon}\Bigr)
%\chi_{F_1}\Bigl(\frac{x}{\varepsilon}\Bigr)
\Bigr\Vert_{L^2({\mathbb R}^d)}<C\varepsilon,
\eeq{Zhikovestimate}
where 
%$\chi_{F_1}$ 
$\widetilde{v}_0$ is the $Q$-periodic extension of the function $v_0=v_0(x,y)$ after setting it to zero for $y\in Q_1.$
%is the indicator function 
In the estimate (\ref{Zhikovestimate}) the constant $C=C(f)>0$ is independent of $\varepsilon,$ but it can not be replaced by
$C\Vert f\Vert_{L^2({\mathbb R}^d)}$ with a constant $C$ that is independent of both $\varepsilon$ and $f.$ (In other words, there are sequences $f^\varepsilon$ that are bounded in $L^2({\mathbb R}^d)$ and are such that $C(f^\varepsilon)\to\infty$ 
as $\varepsilon\to 0.$)

The estimate (\ref{Zhikovestimate}) can also be written in the form
\beq
\Bigl\Vert ({\mathcal A}^\varepsilon+I)^{-1}f-{\mathcal S}^\varepsilon({\mathcal A}^{\rm dp}+I)^{-1}f\Bigr\Vert_{L^2({\mathbb R}^d)}<C(f)\varepsilon,
\eeq{Zhikovres}
where in the expression $({\mathcal A}^{\rm dp}+I)^{-1}f$ the function $f$ is treated as an element of $L^2({\mathbb R}^d\times Q),$  and the operator ${\mathcal S}^\varepsilon: 
L^2({\mathbb R}^d\times Q)\to L^2({\mathbb R}^d)$ is defined by $({\mathcal S}^\varepsilon u)(x)=u(x,x/\varepsilon),$ $x\in{\mathbb R}^d.$
The inequality (\ref{Zhikovres}), however, can not be upgraded to an operator-norm resolvent type statement, in view of the fact that 
the difference of the corresponding spectral projections on a neighbourhood of any point of the form $(\lambda_\infty+I)^{-1},$ where $\lambda_\infty$ is such that $\beta(\lambda)\to\infty$ as $\lambda\to\lambda_\infty,$ does not go to zero in the operator norm as $\varepsilon\to 0.$ 
(Such points $\lambda_\infty$ are the eigenvalues of the operator ${\mathcal A}_0$ that have at least one eigenfunction with non-zero integral over $Q.$) 
Our estimate (\ref{mainest}) therefore rectifies this drawback and captures the operator-norm resolvent asymptotic behaviour of  the sequence ${\mathcal A}^\varepsilon.$

\section*{Acknowledgements}

This work was carried out at Cardiff University under the financial support of the Leverhulme Trust Grant RPG--167 ``Dissipative and non-self-adjoint problems''.


\begin{thebibliography}{99}
\bibitem{Allaire}
Allaire, G. 1992.
Homogenization and two-scale convergence, {\it SIAM J. Math. Anal.} {\bf 23}, 1482--1518.
\bibitem{BLP}
Bensoussan, A., Lions, J.-L., and Papanicolaou, G. C., 1978. {\it Asymptotic Analysis for Periodic Structures,} North-Holland
\bibitem{BS} 
Birman, M. Sh., and  Suslina, T. A.,  2004. Second order periodic differential operators. Threshold properties and homogenisation. {\it St. Petersburg. Math. J.} {\bf 15} (5), 639--714.
\bibitem{CV}
Conca, C., Vanninathan, M., 1997. Homogenisation of periodic structures via Bloch decomposition. {\it SIAM J. Appl. Math.} {\bf 57}, 1639--1659.
\bibitem{Cooper}
Cooper, S., 2012.
Two-scale homogenisation of partially degenerating PDEs with applications to photonic crystals and elasticity, {\it PhD Thesis,} University of Bath.
\bibitem{HL}
Hempel, R.,  and Lienau, K., 2000.  Spectral properties of periodic media in the large coupling limit. {\it Commun. Partial Differ. Equations} {\bf 25}, 1445�1470.
\bibitem{JKO} 
Jikov, V. V., Kozlov, S. M., and Oleinik, O. A. 1994. {\it Homogenization of differential operators and 
integral functionals}, Springer 
\bibitem{SmKam}
Kamotski, I. V., Smyshlyaev, V. P., 2011. Homogenisation of degenerate PDEs and applications to localisation of 
waves, {\it Preprint.}
\bibitem{Kenig}
Kenig, C. E., Lin, F., Shen, Z., 2012. Convergence rates in $L^2$ for elliptic homogenization problems. {\it Archive for Rational Mechanics and Analysis} {\bf 203}(3), 1009-1036.
\bibitem{LL}
Landau, L. D., Lifshitz, E. M., 1960. {\it Electrodynamics of Continuous Media,} Pergamon Press.
\bibitem{Suslina}
Suslina, T. A., 2012. Homogenization of the elliptic Dirichlet problem: operator error estimates in $ L_2$. {\it arXiv preprint arXiv:1201.2286}.
\bibitem{Zhikov1989}
Zhikov, V. V., 1989. Spectral approach to asymptotic problems in diffusion. {\it Diff. Equations} {\bf 25}, 33--39.
\bibitem{Zhikov2000}
Zhikov, V. V. 2000. On an extension of the method of two-scale convergence and its applications, {\it Sb. Math.}, 
{\bf 191}(7), 973--1014.
\bibitem{Zhikov}
Zhikov, V. V., 2005. On gaps in the spectrum of some divergence elliptic operators with periodic coefficients. {\it St. Petersburg Math. J.} {\bf 16} (5) 773--719.
\bibitem{ZP}
Zhikov, V. V., Pastukhova, S. E., 2005. On operator estimates for some problems in homogenization theory. {\it Russian Journal of Mathematical Physics} {\bf 12} (4), 515.
\end{thebibliography}
\end{document}